%% file: final-submission-kindler-v2.tex
\documentclass[reqno,  twoside, a4paper, 11pt]{amsart}
\usepackage[english]{babel}
\usepackage{amsmath}
\usepackage[colorlinks=false,draft=false]{hyperref}
\usepackage{amsfonts}
\usepackage{amssymb}
\usepackage{amsthm}
\usepackage{fixme}
\usepackage{stmaryrd}
\usepackage{microtype}
\usepackage{mathrsfs}
\usepackage[all]{xy}
\usepackage{enumitem}

\DeclareMathOperator{\image}{im}

\DeclareMathOperator{\pr}{pr}
\DeclareMathOperator{\rs}{rs}
\DeclareMathOperator{\tame}{tame}
\DeclareMathOperator{\codim}{codim}
\DeclareMathOperator{\Vect}{Vect}
\DeclareMathOperator{\cont}{cont}
\DeclareMathOperator{\Vectf}{Vectf}
\DeclareMathOperator{\Repf}{Repf}

\DeclareMathOperator{\Spec}{Spec}
\DeclareMathOperator{\Exp}{Exp}

\DeclareMathOperator{\ord}{ord}
\DeclareMathOperator{\HHom}{\mathscr{H}om}
\DeclareMathOperator{\End}{End}

\DeclareMathOperator{\Isom}{Isom}
\DeclareMathOperator{\Strat}{Strat}
\DeclareMathOperator{\Coh}{Coh}
\DeclareMathOperator{\Rep}{Rep}
\DeclareMathOperator{\Gal}{Gal}
\DeclareMathOperator{\rank}{rank}
\DeclareMathOperator{\id}{id}
\DeclareMathOperator{\triv}{triv}
\DeclareMathOperator{\Ind}{Ind}

\DeclareMathOperator{\EExt}{\mathscr{E}xt}

\DeclareMathOperator{\et}{\text{\'et}}
\newcommand{\Z}{\mathbb{Z}}

\newcommand{\Q}{\mathbb{Q}}
\newcommand{\N}{\mathbb{N}}
\newcommand{\C}{\mathbb{C}}
\newcommand{\A}{\mathbb{A}}

\renewcommand{\P}{\mathbb{P}}
\renewcommand{\subset}{\subseteq}

\theoremstyle{theorem}
\newtheorem{Theorem}{Theorem}[section]
\newtheorem{Proposition}[Theorem]{Proposition}
\newtheorem{Lemma}[Theorem]{Lemma}

\newtheorem{Corollary}[Theorem]{Corollary}

\theoremstyle{definition} 
\newtheorem{Definition}[Theorem]{Definition}
\newtheorem{Definition*}[]{Definition}

\newtheorem{Remark}[Theorem]{Remark}

\setenumerate{label=\emph{(\alph*)}, ref=(\alph*)}

\title{Regular singular stratified bundles and tame ramification}
\author{Lars Kindler}
\thanks{This work was supported by the Sonderforschungsbereich/Transregio 45 ``Periods,
moduli spaces and the arithmetic of algebraic varieties'' of the DFG} 

\begin{document}
\begin{abstract}Let $X$ be a smooth variety over an algebraically closed field $k$ of
	positive characteristic. We
	define and study a general notion of regular singularities for stratified bundles
	(i.e.~$\mathcal{O}_X$-coherent $\mathscr{D}_{X/k}$-modules)
	on $X$ without relying on resolution of singularities. The main result is that the
	category of regular singular stratified bundles with finite monodromy is
	equivalent to the category of continuous representations of the tame fundamental
	group on finite
	dimensional $k$-vector spaces. As a corollary we obtain that a stratified bundle with finite monodromy is
	regular singular if and only if it is regular singular along all curves mapping to
	$X$.\end{abstract}
\maketitle
\section{Introduction}
If $X$ is a smooth, connected, complex variety, then it is an elementary fact that a
vector bundle with flat connection and finite monodromy at a closed  point $x\in X(\C)$ is
automatically regular singular. This implies, together with the Riemann-Hilbert correspondence as developed in
\cite{Deligne/RegularSingular}, that the category of flat connections with finite
monodromy is equivalent to the category $\Repf_{\C}^{\cont}\pi_1^{\et}(X,x)$
of continuous representations of the \'etale fundamental group of $X$ on
finite dimensional complex vector spaces equipped with the discrete topology.
The goal of
this article is to establish an analogous statement over an algebraically closed field $k$ of
positive characteristic $p$.

If $X$ is a smooth, connected $k$-variety, then both the category of vector bundles with flat
connection and the category of coherent $\mathcal{O}_X$-modules with flat connection lack
many of the nice features that they have over the complex numbers. In particular they are not
Tannakian categories over $k$. To remedy this fact we pass to the perspective of
modules over the ring of differential operators $\mathscr{D}_{X/k}$ on $X$: Over $\C$, the
category of vector bundles with flat connection
is equivalent to the category of left-$\mathscr{D}_{X/\C}$-modules which are coherent as
$\mathcal{O}_X$-modules. It turns out that if $X$ is a smooth $k$-variety with $k$ of
positive characteristic, then the category of $\mathcal{O}_{X}$-coherent $\mathscr{D}_{X/k}$-modules
is Tannakian over $k$. This was already worked out in \cite{Saavedra}.
Following Grothendieck
and Saavedra we call an $\mathcal{O}_{X}$-coherent $\mathscr{D}_{X/k}$-module a
\emph{stratified bundle} (see Remark \ref{rem:names}).

The content of this article is the definition and study of a sensible notion of regular singularity for
stratified bundles in positive characteristic. The main result is the following:

\begin{Theorem}\label{thm:mainINTRO}Let $k$ be an algebraically closed field of positive
	characteristic. If $X$ is a smooth, connected $k$-variety, and $x\in
	X(k)$ a rational point,  then a
	stratified bundle $E$ on $X$ is regular singular with finite monodromy at $x$ if
	and only if $E$ is trivialized on a finite tame covering.
\end{Theorem}
The notion of tameness used here is due to Wiesend and extensively studied in
\cite{Kerz/tameness}. From Theorem \ref{thm:mainINTRO} we easily obtain the following,
pleasantly conceptual statement:

\begin{Corollary}\label{cor:repsINTRO}If $X$ is a smooth, connected $k$-variety, then after
choice of a base point $x\in X(k)$, the fiber $E|_x$ of a regular singular stratified bundle $E$ carries a
functorial $\pi_1^{\tame}(X,x)$-action, and this functor induces an equivalence of
categories between the category of regular singular stratified bundles with finite monodromy at
$x$ and the category $\Repf_k^{\cont}\pi_1^{\tame}(X,x)$ of finite dimensional continuous
$k$-representations of $\pi_1^{\tame}(X,x)$.
\end{Corollary}

From the main result of \cite{Kerz/tameness} we obtain our third main result:
\begin{Corollary}\label{cor:curvesINTRO}
	If $X$ is a smooth, connected $k$-variety, then a stratified bundle $E$ on $X$
	with finite monodromy is regular singular, if and only if it is regular singular
	along all curves on $X$.
\end{Corollary}

We make a few remarks about the definition of the notion of regular
singularity and give a brief outline of the article: If the $k$-variety
$X$ has a good compactification, i.e.~if there exists a smooth, proper
$k$-variety $\overline{X}$, such that $X$ is a dense open subset of
$\overline{X}$ with a strict normal crossings divisor as complement, then the
notion of regular singularity is formally the same as in characteristic $0$: A
stratified bundle $E$ is regular singular if and only if it extends to an
$\mathcal{O}_{\overline{X}}$-coherent, $\mathcal{O}_{\overline{X}}$-torsion free
$\mathscr{D}_{\overline{X}/k}(\log \overline{X}\setminus X)$-module. This
definition was explored in \cite{Gieseker/FlatBundles}.
Unfortunately, in characteristic $p>0$, it is not known whether every $X$
admits a good compactification. One first approach to a general definition of
regular singularity could be to use \cite{deJong/Alterations} to replace $X$
by an alteration
$Y\rightarrow X$, such that $Y$ admits a good compactification.
Unfortunately, de Jong's theorem and its refinements (e.g.~Gabber's
prime-to-$\ell$ alterations) do not provide control over the wild part of the
alteration $Y\rightarrow X$, which implies that one cannot use $Y$ to check
stratified bundles on $X$ for regular singularities\footnote{Purely
	inseparable alterations on the other hand could be very helpful, see
	\cite[Question 2.9]{AbramovichOort/Alterations}.}. For example, an Artin-Schreier covering
$f:\A^1_k\rightarrow \A^1_k$ is an alteration and
$f^*f_*\mathcal{O}_{\A^1_k}=\mathcal{O}_{\A_k^1}^{\oplus p}$ is regular singular, but
$f_*\mathcal{O}_{\A^1_k}$ is not. 

Instead, we use the fact that regular singularity should be a notion local in
the codimension $1$ points at ``infinity'', and we define a stratified bundle to
be regular singular, if it is regular singular with respect to all good
\emph{partial} compactifications of $X$, see Section \ref{sec:rs}.

In Section \ref{sec:finitebundles} we establish the basic facts about
stratified bundles and monodromy groups. Section \ref{sec:logdiffops} provides
the technical basis for the discussion of regular singularities, which starts
in Section \ref{sec:rs}. In Section \ref{sec:tauextension} we prove an
extension result for regular singular stratified bundles, which is exploited
in Section \ref{sec:rsfinite}. Here we prove Theorem \ref{thm:mainINTRO} with
respect to a fixed good partial compactification. In Section
\ref{sec:rsingeneral} and \ref{sec:curves} we finally complete the proofs of
Theorem \ref{thm:mainINTRO} and Corollary \ref{cor:curvesINTRO}.

To close this introduction, we make a remark about the assumption that
the base field $k$ be algebraically closed. The argument at the end of the
proof of \cite[Prop.~2.4]{Esnault/Gieseker} shows that for $X$ proper over
$k$,  a stratified bundle $E$
on $X$ is trivial if and only if $E\otimes_k \bar{k}$ is a trivial stratified
bundle (with respect to $\bar{k}$) on $X\times_k \bar{k}$. This shows that 
Corollary \ref{cor:repsINTRO} is false even for proper $X$, for example for $X=\Spec k$, or
$X=\P^r_k$, if $k$ is not algebraically closed (this was also noted in
\cite[2.4]{DosSantos}). Since establishing Corollary \ref{cor:repsINTRO} as
direct analogue to the classical situation over the complex numbers is the
main motivation behind writing this article, we allow ourselves to always
assume $k$ to be algebraically closed.

\vspace{.5cm}
\noindent\emph{Acknowledgements:} The results contained in this article are
part of the author's
dissertation \cite{Kindler/thesis}, written under the supervision of H\'el\`ene
Esnault. The author is extremely grateful for the support, kindness, and generosity 
which Prof.~Esnault offered during his graduate studies, and continues to
offer. Additionally, the author wishes to express heartfelt thanks to Prof.~Alexander
Beilinson and the University of Chicago for the opportunity to visit Chicago in February
and March 2012. Large parts of the arguments written down here were developed during this
stay.

\vspace{.5cm}

\noindent\emph{Notation:} If $k$ is a field and $G$ an affine $k$-group scheme, then we
write $\Vectf_k$ (resp.~$\Repf_k G$) for the category of \emph{finite dimensional} $k$-vector
spaces (resp.~representations of $G$ on finite dimensional $k$-vector spaces).
Similarly, $\Vect_k$ (resp.~$\Rep_k G)$ denotes the category of all $k$-vector spaces
(resp.~all representations of $G$ on $k$-vector spaces).
If $\mathcal{T}$ is a Tannakian category over $k$, and
$\omega:\mathcal{T}\rightarrow \Vectf_k$ a fiber functor, then we write
$\pi_1(\mathcal{T},\omega)$ for the affine $k$-group scheme such that $\omega$
induces an equivalence $\mathcal{T}\rightarrow \Repf_k
\pi_1(\mathcal{T},\omega)$ (\cite[Thm.~2.11]{DeligneMilne}).

If $k$ is a field and $X$ an affine $k$-scheme, then we say that global
sections $x_1,\ldots, x_n\in H^0(X,\mathcal{O}_{X})$ are \emph{coordinates for
$X$} if the morphism $X\rightarrow \A^n_k$ defined by them is \'etale.

\section{Stratified Bundles and Monodromy}\label{sec:finitebundles}
In this section we collect some facts about the category of stratified bundles and its
Tannakian properties, and we recall the notion of the monodromy group of a stratified bundle. In
all of this section, we fix an algebraically closed field $k$ of positive characteristic
$p$, and $X$ will always denote a smooth, connected, separated $k$-scheme of finite type.

\begin{Definition}\label{defn:strat} We write
	$\Strat(X)$ for the category of \emph{stratified bundles on $k$}, i.e.~for the
	category of left-$\mathscr{D}_{X/k}$-modules which are coherent as
	$\mathcal{O}_X$-modules (when considering sections of $\mathcal{O}_X$ as
	differential operators of order $0$),
	together with  $\mathscr{D}_{X/k}$-linear morphisms.  Here $\mathscr{D}_{X/k}$ is
	the sheaf of rings of differential operators of $X$ relative to $k$, as developed in
	\cite[\S16]{EGA4}.
	A stratified bundle is called \emph{trivial}, if it is isomorphic to
	$\mathcal{O}_X^n$ together with its canonical diagonal
	left-$\mathscr{D}_{X/k}$-action.

	If $E$ is a stratified bundle, we write $E^{\nabla}$ for the sheaf of \emph{horizontal
	sections} of $E$, defined by 
	\[E^{\nabla}(U)=\{s\in E(U)| \theta(s)=0 \text{ for all
	}\theta\in\mathscr{D}_{X/k}(U)\text{ with } \theta(1)=0\}.\]
	Note for example that $(\mathcal{O}_{X}^n)^{\nabla}(U)=k^n$ for all $U\subset X$
	open.
\end{Definition}
\begin{Remark}\label{rem:names}Since $X$ is smooth over $k$, a stratified bundle in our
	sense is precisely a coherent $\mathcal{O}_X$-module equipped with a
	``stratification'' in the sense of \cite{Grothendieck/Crystals}.
	Moreover a stratified bundle is automatically a locally free
	$\mathcal{O}_X$-module, see e.g. \cite[2.17]{BerthelotOgus/Crystalline}, so the
	name ``stratified bundle'' is at least historically appropriate.
\end{Remark}
\begin{Proposition}[{\cite[\S VI.1]{Saavedra}}]\label{prop:stratistannakian}
	The category
	$\Strat(X)$ is a $k$-linear Tannakian category, and a rational point $x\in X(k)$
	gives a neutral fiber functor \[\omega_x:\Strat(X)\rightarrow \Vectf_k, E\mapsto
	E|_x:=x^*E.\]
\end{Proposition}
\begin{Definition}\label{defn:fgcat}
	Write
	$\left<E\right>_{\otimes}$ for the smallest full sub-Tannakian subcategory of $\Strat(X)$
	containing a stratified bundle $E\in\Strat(X)$. The objects of $\left<E\right>_{\otimes}$ are subquotients of
	objects of the form
	$P(E,E^{\vee})$, with $P(r,s)\in \N[r,s]$. 
\end{Definition}

We analyze how the above categories behave when passing to an open subscheme
of $X$:
\begin{Lemma}\label{lemma:restriction}
	Let $U$ be an open dense subscheme of $X$. The following statements are true:
	\begin{enumerate}
		\item\label{item:restrictionRegular1} If $E\in \Strat(X)$, then the restriction functor
			$\rho_{U,E}:\left<E\right>_{\otimes}\rightarrow
			\left<E|_{U}\right>_{\otimes}$ is an equivalence.
		\item\label{item:restrictionRegular2} The restriction functor $\rho_U:\Strat(X)\rightarrow \Strat(U)$ is fully
			faithful.        
		\item\label{item:restrictionRegular4} If $\codim_X(X\setminus U)\geq 2$, then $\rho_U:\Strat(X)\rightarrow
			\Strat(U)$ is an equivalence.~
	\end{enumerate}
\end{Lemma}                                                                                       
\begin{proof} By our standing assumptions, $X$ is connected. Clearly
	\ref{item:restrictionRegular2} follows directly from
	\ref{item:restrictionRegular1}.
	We first prove \ref{item:restrictionRegular4}.
	Assume $\codim_{X}(X\setminus U)\geq 2$. Denote by $j:U\hookrightarrow X$ the open
	immersion. Let $E$ be a stratified bundle on $U$, in particular a locally free,
	finite rank $\mathcal{O}_U$-module. Then $j_*E$ is 
	$\mathcal{O}_{X}$-coherent by the assumption
	on the codimension (\cite[Exp.~VIII, Prop.~3.2]{SGA2}), and it carries a
	$\mathscr{D}_{X/k}$-action, since $j_*\mathscr{D}_{U/k}=\mathscr{D}_{X/k}$. Thus it is also
	locally free. If $\overline{E}'$ is any other locally free extension of $E$ to
	$X$, then we get a short exact sequence 
	\[0\rightarrow \overline{E}'\hookrightarrow j_*E\rightarrow G\rightarrow 0,\]
	with $G$ supported on $X\setminus U$. Since $X$ is smooth, we have
	$\HHom_{\mathcal{O}_X}(G,\mathcal{O}_X)=0$ and
	$\EExt_{\mathcal{O}_X}^1(G,\mathcal{O}_X)=0$ by \cite[Exp.~III, Prop.~3.3]{SGA2}.
	This implies that there
	is a canonical isomorphism
	$\overline{E}'\cong ((j_*E)^{\vee})^\vee=j_*E$, compatible with the
	$\mathscr{D}_{X/k}$-structures.  It follows that the functors
	$j_*$ and $j^*$ are quasi-inverse to each other, which proves that $j^*$ is an
	equivalence.

	To prove \ref{item:restrictionRegular1}, let $\codim_{X}(X\setminus U)$ be arbitrary.
	Let us first show that $\rho_{U,E}:\left<E\right>_{\otimes}\rightarrow
	\left<E|_U\right>_{\otimes}$ is essentially surjective. 	If $F$ is an object of $\left<E\right>_{\otimes}$,
        and $F'$ a subobject of $\rho_U(F)=F|_U$, then $j_*F'\subset j_*(F|_U)$. The
        quasi-coherent $\mathcal{O}_X$-modules $j_*F'$ and $j_*(F|_U)$ carry
        $\mathscr{D}_{X/k}$-actions, such that $j_*F'\subset j_*(F|_U)$  and $F\subset
	j_*(F|_U)$ are sub-$\mathscr{D}_{X/k}$-modules.
	Then $F'_X:=j_*F'\cap F$ is an $\mathcal{O}_{{X}}$-coherent
	$\mathscr{D}_{{X}/k}$-submodule
        of $F$ extending $F'$. Thus we have seen that the essential image of $\rho_U$ is closed
	with respect to taking subobjects.

	But this also shows that $F|_U/F'$ can be extended to $X$: We just saw that $F'$
	extends to $F'_X\subset F$, and since $\rho_U$ is exact, this means that
	$(F/F'_X)|_U=F|_U/F'$. It follows that the essential image of $\rho_U$ is
	closed with respect to taking subquotients.

	The objects of
	$\left<E|_U\right>_{\otimes}$ are subquotients of stratified bundles bundles of the form
	$P(E|_U,(E|_U)^{\vee})$, with $P(r,s)\in \N[r,s]$. We can lift all of the objects
	$P(E|_U,(E|_U)^\vee)$ to $X$, since $(E^{\vee})|_U=(E|_U)^{\vee}$, so $\rho_U$ is
	essentially surjective.

        It remains to check that
	$\rho_{U,E}:\left<E\right>_{\otimes}\rightarrow\left<E|_U\right>_{\otimes}$ is fully
	faithful. Let $F_1,F_2\in \left<E\right>_{\otimes}$. 
	Since
	$\hom_{\Strat(X)}(F_1,F_2)=\hom_{\Strat(X)}(\mathcal{O}_X,F_1^{\vee}\otimes F_2)$,
	and similarly over $U$, 
	we may replace $F_1$ by $\mathcal{O}_X$. Moreover,  by
	\ref{item:restrictionRegular4}, we may remove
	closed subsets of codimension $\geq 2$ from $X$, so without loss of
	generality we may assume that $X\setminus U$ is the support of a 
	smooth irreducible divisor $D$ with generic point $\eta$. 
	To finish the proof, we may shrink $X$ around $\eta$. 	
	Thus we may
	assume that $X=\Spec A$ for some finite type $k$-algebra $A$, that $F_2$ corresponds to a free
	$A$-module, say with basis
	$e_1,\ldots, e_n$, and that  we have global coordinates
	$x_1,\ldots, x_n$, such that $D=(x_1)$.  Then $U=\Spec A[x_1^{-1}]$.
	Finally assume that $\phi:\mathcal{O}_U\rightarrow F_2|_{U}$ is a morphism of
	stratified bundles given by $\phi(1)=\sum_{i=1}^n \phi_i e_i$, $\phi_i\in
	A[x_1^{-1}]$. 	We get
	\[0 =\partial_{x_1}(\phi(1))=
	\sum_{i=1}^n\partial_{x_1}(\phi_i)e_i+\phi_i\underbrace{\partial_{x_i}(e_i)}_{\in
	\image(F_2\rightarrow F_2\otimes A[x_1^{-1}])}\]
	and in particular that $\partial_{x_1}(\phi_i)\in \phi_iA\subset A[x_1^{-1}]$.
	By induction, we assume 
	$\partial_{x_1}^{(a)}(\phi_i)\in \phi_i A$
	for every $a\leq m$.
	But then we compute
	\[ 0 = \partial_{x_1}^{(m)}(\phi(1))=\sum_{i=1}^n
	\sum_{\begin{subarray}{c}a+b=m\\a,b\geq 0\end{subarray}}
	\partial_{x_1}^{(a)}(\phi_i)\partial_{x_1}^{(b)}(e_i),\]
	to see that $\partial_{x_1}^{(m)}(\phi_i)\in \phi_i A$, so this holds for all $m\geq 1$. In
	particular we see that the pole order of $\partial_{x_1}^{(m)}(\phi_i)$ along $D$
	is at most the pole order of $\phi_i$ along $D$, for all $m\geq 0$. 
        This shows that $\phi_i\in
	A$ for every $i$, so $\phi$ is defined over all of $X$.

	It also extends uniquely: If $\psi_1,\psi_2$ are morphisms
	$\mathcal{O}_X\rightarrow F_2$, such that $(\psi_1-\psi_2)\otimes A[x_1^{-1}]=0$,
	then $\psi_1=\psi_2$, as $A$ is an integral domain and $F_2$ is torsion-free. This
	finishes the proof of \ref{item:restrictionRegular1}.
\end{proof}
\begin{Corollary}
	If $E$ is a stratified bundle, then $E^{\nabla}$ is the constant sheaf associated
	with a finite dimensional $k$-vector space. We will usually identify $E^{\nabla}$
	with this vector space.
\end{Corollary}
\begin{proof}
	A horizontal section of $E$ over some open $U$ spans a trivial substratified
	bundle of $E|_U$, which lifts to $X$ by the lemma.
\end{proof}

Now we proceed to the notion of monodromy, which is completely analogous to the classical
notion in characteristic $0$.

\begin{Definition}\label{defn:monodromy}
If $\omega:\left<E\right>_{\otimes}\rightarrow \Vectf_k$ is a fiber functor, we denote by
$G(E,\omega):=\pi_1(\left<E\right>_{\otimes},\omega)$ the \emph{monodromy group of $E$ with respect to $\omega$}, i.e.~the
	affine, finite type $k$-group scheme
	associated via Tannaka duality with the Tannakian category
	$\left<E\right>_{\otimes}\subset \Strat(X)$ and the fiber functor $\omega$.
\end{Definition}
From Lemma \ref{lemma:restriction}  we immediately derive (using
\cite[Prop.~2.21]{DeligneMilne}):
\begin{Lemma}\label{lemma:restrictionMonodromy}Let $U\subset X$ be an open dense subscheme, then: 
	\begin{enumerate}
		\item If $\omega:\Strat(U)\rightarrow \Vectf_k$ is a fiber functor, then
			the induced morphism
			$\pi_1(\Strat(U),\omega)\rightarrow\pi_1(\Strat(X),\omega(
			(-)|_U))$ is faithfully flat.
		\item If $E$ is a stratified bundle, and
			$\omega:\left<E|_U\right>_{\otimes}\rightarrow \Vectf_k$ a fiber
			functor, then the induced morphism $G(E|_U,\omega)\rightarrow
			G(E,\omega( (-)|_U))$ is an isomorphism of $k$-group schemes.
	\end{enumerate}
\end{Lemma}

The following theorem (relying on the assumption that $k$ is algebraically closed) is
essential for what is to follow:
\begin{Theorem}[{\cite[Cor.~12]{DosSantos}}]\label{thm:dosSantos}
	If $E\in \Strat(X)$ and $\omega:\Strat(X)\rightarrow  \Vectf_k$ a
	fiber functor, then the $k$-group scheme $G(E,\omega)$ is smooth.
\end{Theorem}

\begin{Remark}In \cite[Cor.~12]{DosSantos}, the above theorem is only formulated in the
	case that $G(E,\omega)$ is finite over $k$, but the given proof holds more
	generally. Indeed, by \cite[Thm.~11]{DosSantos}, the affine group scheme
	$\pi_1(\Strat(X),\omega)$ is perfect and thus reduced. This means that any
	quotient of $\pi_1(\Strat(X),\omega)$ is also reduced.
\end{Remark}

\begin{Definition}
	An object $E\in
	\Strat(X)$ is said to be \emph{finite} or to have \emph{finite monodromy} if there
	is a fiber functor $\omega:\left<E\right>_{\otimes}\rightarrow \Vectf_k$, such that
	$G(E,\omega)$ is finite over $k$. Since $k$ is algebraically closed, Theorem
	\ref{thm:dosSantos} implies that this is equivalent to $G(E,\omega)$ being a
	constant $k$-group scheme associated with a finite group.
\end{Definition}
\begin{Remark}
	Note that in our situation $\left<E\right>_{\otimes}$ always has $k$-linear fiber
	functors by \cite[Cor.~6.20]{Deligne/Festschrift}. Thus $E$ is finite if and only if every object of $\left<E\right>_{\otimes}$
	is isomorphic to a subquotient of $E^{\oplus n}$ for some $n$, see \cite[Prop.
	2.20]{DeligneMilne}. In particular, $E$ has finite monodromy if and only if
	$G(E,\omega)$ is finite for every fiber functor
	$\omega:\left<E\right>_{\otimes}\rightarrow \Vectf_k$. 
\end{Remark}
\begin{Proposition}\label{prop:restriction}
	If $E$ is a
	stratified bundle on $X$, then the following statements are equivalent:
	\begin{enumerate}
		\item $E$ is finite,
		\item $E|_U$ is finite for some open dense $U\subset X$,
		\item $E|_U$ is finite for every open dense $U\subset X$.
	\end{enumerate}
\end{Proposition}
\begin{proof}
	This follows directly from Lemma \ref{lemma:restrictionMonodromy}.
\end{proof}

\begin{Remark}\label{rem:tannaka}Before we can state the next proposition, we have to recall some facts about
	Tannakian categories over $k$; general references are \cite{Saavedra},
	\cite{DeligneMilne}, and \cite{Deligne/Festschrift}.
	\begin{enumerate}[label=(\alph*), ref=(\alph*)]
		\item The category $\Vectf_k$  of finite dimensional $k$-vector spaces is the Tannakian category associated with the trivial
			group scheme.
		\item If $\mathcal{T}$ is a Tannakian category, then there exists a unique
			strictly full sub-Tannakian category $\mathcal{T}_{\triv}\cong
			\Vectf_k$
			whose objects are the trivial objects of $\mathcal{T}$, and there
			exists a functor $H^0(\mathcal{T},-):\mathcal{T}\rightarrow
			\Vectf_k$ assigning to every object of $\mathcal{T}$ its maximal
			trivial subobject (which exists, because every object
			of $\mathcal{T}$ has finite length by
			\cite[2.13]{Deligne/Festschrift}). The functor
			$H^0(\mathcal{T}, -)$ extends to a functor
			$H^0(\Ind(\mathcal{T}),-):\Ind(\mathcal{T})\rightarrow \Vect_k$ on the associated
			$\Ind$-categories.
		\item\label{tannakaRightReg} Affine $k$-group schemes are always projective limits of finite
			type $k$-group schemes
			(\cite[Cor.~2.4]{DeligneMilne}), and if $G$ is an
			affine $k$-group scheme, then 
			$\Ind(\Repf_k G)=\Rep_k G$
			(\cite[Cor.~2.7]{DeligneMilne}). We have $H^0(\Repf_k
			G, - )= ( - )^{G}$ and $H^0(\Rep_k G, - )=(-)^G$. 
			
			In the $\Ind$-category $\Ind(\Repf_k
			G)=\Rep_k G$ there is the right-regular
			representation $(\mathcal{O}_G,\Delta)$, where
			$\Delta:\mathcal{O}_G\rightarrow \mathcal{O}_G\otimes_k
			\mathcal{O}_G$ is the diagonal comultiplication. This
			representation is an algebra over the trivial representation,
			which also is its maximal trivial object (because $G$
			acting on itself does not have any invariants), and
			it has the property that there exists a functorial isomorphism
			\[(\mathcal{O}_G,\Delta)\otimes_k V\cong
			(\mathcal{O}_G,\Delta)^{\rank V}\] for every representation $V\in
			\Repf_k G$ (\cite[2.3.2 (d)]{DosSantos}). In other words, the composition of functors $V\mapsto
			(\mathcal{O}_G\otimes V)^{G}$ is the forgetful functor $\Repf_k
			G\rightarrow \Vectf_k$.
		\item\label{rem:tannakaD} If $\mathcal{T}$ is a Tannakian
			category over $k$, and $\omega:\mathcal{T}\rightarrow \Vectf_k$ a fiber functor, then
			$\omega$ induces a $\otimes$-equivalence
			$\mathcal{T}\xrightarrow{\cong}\Repf_k \pi_1(\mathcal{T},\omega)$
			with $G:=\pi_1(\mathcal{T},\omega)$ an affine
			$k$-group scheme, and such that $\omega$ is naturally
			isomorphic to the composition
			$\mathcal{T}\xrightarrow{\cong}\Repf_k G
			\xrightarrow{\text{forget}}
			\Vectf_k$ (\cite[Thm.~2.11]{DeligneMilne}).
			The right-regular representation $(\mathcal{O}_G,\Delta)$
			corresponds via $\omega$ to an object $A_{\omega}\in
			\Ind(\mathcal{T})$, which is an algebra over the unit object of
			$\mathcal{T}$, and such that the functor
			$\omega:\mathcal{T}\rightarrow \Vectf_k$ can be written as
			$T\mapsto H^0(\Ind(\mathcal{T}),T\otimes A_{\omega})$.
	\end{enumerate}
\end{Remark}
\begin{Proposition}\label{prop:universaltorsor}
	Let $E\in \Strat(X)$ be a stratified bundle, and
	$\omega:\left<E\right>_{\otimes}\rightarrow \Vectf_k$ a fiber functor. Write 
	$G:=G(E,\omega)$.	Then there exists a smooth $G$-torsor
	$h_{E,\omega}:X_{E,\omega}\rightarrow X$ with the following
	properties:
	\begin{enumerate}
		\item\label{prop:universaltorsorMAIN}Every object of
			$\left<E\right>_{\otimes}$ has
			finite monodromy if and only if 
			$h_{E,\omega}$ is finite \'etale.
	\end{enumerate}
	From now on assume that $G$ is a finite (thus constant by
	Theorem \ref{thm:dosSantos}) group scheme, and hence $h_{E,\omega}$
	finite \'etale. Then $h_{E,\omega}$ has the following properties.	
	\begin{enumerate}[resume*]
		\item\label{prop:universaltorsorA} An object $E'\in\Strat(X)$ is contained
			in $\left<E\right>_{\otimes}$ if and only
			if $h_{E,\omega}^*E'\in\Strat(X_{E,\omega})$ is trivial.
		\item\label{prop:universaltorsorB} If $\mathcal{S}$ is a strictly
			full sub-Tannakian
			subcategory of $\Strat(X)$ such that
			$\mathcal{S}\subset \left<E\right>_{\otimes}$, then there exists $E'\in
			\left<E\right>_{\otimes}$ such that
			$\mathcal{S}=\left<E'\right>_{\otimes}$, and a finite \'etale
			morphism $g$ such that the diagram
			\begin{equation*}
				\xymatrix{ 
				X_{E,\omega}\ar[d]_{h_{E,\omega}}\ar[r]^g
				&
				X_{E',\omega}\ar[dl]^{h_{E',\omega}}\\
			 X
				 }
			 \end{equation*}
			 commutes.
		 \item\label{prop:universaltorsorC} If $\omega'$ is a second fiber functor
			 on $\left<E\right>_{\otimes}$, then there is an
			 isomorphism of $X$-schemes
			 $X_{E,\omega}\xrightarrow{\cong}X_{E,\omega'}$.
		 \item\label{prop:universaltorsorD}For $E'\in
			 \left<E\right>_{\otimes}$ we have a
			 functorial isomorphism
			 \[\omega(E')=H^0(\Strat(X_{E,\omega}),h_{E,\omega}^*E')=\left(
			 h_{E,\omega}^*E' \right)^{\nabla}.\]
	 \end{enumerate}
	 Conversely, if $G$ is a finite constant group scheme and $f:Y\rightarrow X$ a
	 $G$-torsor, i.e.~a finite \'etale Galois covering with group $G(k)$, and
	 $\mathcal{Y}$ the sub-Tannakian subcategory of $\Strat(X)$ generated by those objects
	 which become trivial after pullback along $f$, then:
	 \begin{enumerate}[resume*]
		 \item\label{prop:universaltorsorE}  
			 $\mathcal{Y}=\left<f_{*}\mathcal{O}_Y\right>_{\otimes}$ as
			 strictly full subcategories of $\Strat(X)$,  and there is a fiber functor
			 $\omega_f:\mathcal{Y}\rightarrow \Vectf_k$, such that
			 $f=h_{Y,\omega_f}$, and such that
			 $G(\left<f_*\mathcal{O}_Y\right>_{\otimes},\omega_f)=G$.
		 \item\label{prop:universaltorsorF}If $\mathcal{S}\subset \Strat(X)$ is a
			 strictly full sub-Tannakian subcategory such that
			 $f^*E$ is trivial for every object $E\in \mathcal{S}$, then
			 $\mathcal{S}\subset \mathcal{Y}$, so
			 $\mathcal{S}=\left<E\right>_{\otimes}$ for some $E\in
			 \mathcal{Y}$,
			 $G(E,\omega_f|_{\mathcal{S}})$ is finite constant
			 and there
			 is a morphism $g:Y\rightarrow X_{E,\omega}$, such that
			 $f=h_{E,\omega}\circ g$.
	\end{enumerate}
\end{Proposition}
\begin{proof}
	This proposition is fairly well-known, but the author does not know of a complete
	reference. Certainly all the ideas from the theory of Tannakian
	categories are contained in \cite{Deligne/Festschrift} (particularly
	\cite[\S 9.]{Deligne/Festschrift}) and
	\cite{DeligneMilne}.

	We use the the facts recalled in Remark \ref{rem:tannaka}. The main ingredient
	not intrinsic to the theory of Tannakian categories is the fact that if $E\in
	\Strat(X)$ is a stratified bundle then $G(\left<E\right>_{\otimes},\omega)$
	is a smooth $k$-group scheme by dos Santos' theorem \ref{thm:dosSantos}.
	In particular, if $G(\left<E\right>_{\otimes},\omega)$ is finite, then it is
	finite \'etale and hence constant if $k$ is algebraically closed.

	Back to the notations of the proposition: Let $\rho:\Strat(X)\rightarrow \Coh(X)$
	denote the forgetful functor. With the fiber functor
	$\omega:\left<E\right>_{\otimes}\rightarrow \Vectf_k$ we associate in Remark \ref{rem:tannaka}
	an object $A_{E,\omega}$ of $\Ind(\Strat(X))$. Via $\rho$ we can consider
	$A_{E,\omega}$ as a quasi-coherent $\mathcal{O}_X$-algebra with
	$\mathscr{D}_{X/k}$-action, and
	such that $A_{E,\omega}$ corresponds to 
	the right-regular representation of $G$ in $\Rep_k G$. Write $h_{E,\omega}:X_{E,\omega}:=\Spec
	A_{{E},\omega}\rightarrow X$ for this $G$-torsor. The property from
	Remark \ref{rem:tannaka} that
	$(\mathcal{O}_G,\Delta)\otimes V= (\mathcal{O}_G,\Delta)^{\rank V}$ for $V\in
	\Repf_k G$ translates into the property that
	\[X_{E,\omega}:=\Spec
	A_{E,\omega}\cong\underline{\Isom}_k^{\otimes}(\omega(-)\otimes_k
	\mathcal{O}_X,\rho|_{\left<E\right>_{\otimes}}),\]
	where $\underline{\Isom}_k^{\otimes}$ is defined as in
	\cite[1.11]{Deligne/Festschrift}.

	Since $G$ is smooth over $k$, $h_{E,\omega}$ is smooth, and it is finite
	if and only if $A_{E,\omega}$ is coherent, if and only if
	$\left<E\right>_{\otimes}=\left<A_{E,\omega}\right>_{\otimes}$ (because
	$E\subset A_{E,\omega}^m$ for some $m>0$), so \ref{prop:universaltorsorMAIN} follows.

	Now assume that $G$ is a finite \'etale group scheme on $k$, hence constant.
	Then  $h_{E,\omega}$ is a
	finite \'etale morphism; in particular, $A_{E,\omega}$ is an
	$\mathcal{O}_X$-algebra in the category $\Strat(X)$.  Moreover, the
	$\mathscr{D}_{X_{E,\omega}/k}$-action on
	$h_{E,\omega}^*E=E\otimes_{\mathcal{O}_X} A_{E,\omega}$ agrees
	with the tensor product $\mathscr{D}_{X/k}$-action on $E\otimes_{\mathcal{O}_X}
	A_{E, \omega}$ via the isomorphism $\mathscr{D}_{X_{E,\omega}/k}\xrightarrow{\cong}
	h^*_{E,\omega}\mathscr{D}_{X/k}$. In other words, the pull-back functor
	$h_{E,\omega}^*$ ``is''
	\[\Strat(X)\rightarrow \Strat(X_{E,\omega}), E\mapsto E\otimes A_{E,\omega}.\]

	Now everything follows fairly directly from general theory:
	If $E'\in \left<E\right>_{\otimes}$, then $E'\otimes A_{E,\omega}$ is trivial, and
	conversely if, if $E'\otimes A_{E,\omega}=A_{E,\omega}^n$, then the canonical map
	$E'\rightarrow E'\otimes A_{E,\omega}$ exhibits $E'$ as an object of
	$\left<A_{E,\omega}\right>_{\otimes}=\left<E\right>_{\otimes}$. This proves
	\ref{prop:universaltorsorA}.
	
	For \ref{prop:universaltorsorB}, first  note that $\pi_1(\mathcal{S}',\omega)$ is
	a quotient of $G(E,\omega)$, so $\mathcal{S'}$ can be written as
	$\left<E'\right>_{\otimes}$ for some $E'\in \left<E\right>_{\otimes}$. The morphism $g:X_{E,\omega}\rightarrow
	X_{E',\omega}$ then comes from the morphism of right-regular representations
	$(\mathcal{O}_{G(E')},\Delta)\rightarrow (\mathcal{O}_{G(E)},\Delta)$.

	\ref{prop:universaltorsorC} follows from the fact that
	$\underline{\Isom}_k^{\otimes}(\omega,\omega')$ is an fpqc-torsor on $k$, but
	$k$ is algebraically closed, so the torsor is trivial, hence $\omega\cong
	\omega'$.

	Statement \ref{prop:universaltorsorD}, follows from Remark \ref{rem:tannaka},
	\ref{rem:tannakaD}.

	For \ref{prop:universaltorsorE}, note that $f_* \mathcal{O}_Y$ is a stratified
	bundle on $X$; more precisely, it is an $\mathcal{O}_X$-algebra in $\Strat(X)$.
	As above we see that the $\mathscr{D}_{Y/k}$-structure on $f^*E$
	agrees with the $\mathscr{D}_{X/k}$-structure on $E\otimes_{\mathcal{O}_X} f_*\mathcal{O}_Y$ via
	the isomorphism $\mathscr{D}_{Y/k}\xrightarrow{\cong} f^*\mathscr{D}_{X/k}$. Since $f$ is
	Galois, $f_*\mathcal{O}_Y\otimes_{\mathcal{O}_X} f_{*}\mathcal{O}_Y\cong
	f_*\mathcal{O}_Y^{\deg f}$, so $f_*\mathcal{O}_Y$ is trivialized on $Y$.
	Conversely, if $f^*E=E\otimes f_*\mathcal{O}_Y$ is trivial, i.e.~$E\otimes
	f_*\mathcal{O}_Y\cong f_*\mathcal{O}_Y^n$ for some $n$, then $E$ is a
	$\mathscr{D}_{X/k}$-submodule of $f_{*}\mathcal{O}_Y^n$, so we have proven that the
	subcategory of $\Strat(X)$ spanned by bundles $E$ trivialized on $Y$ is precisely
	$\left<f_*\mathcal{O}_Y\right>_{\otimes}$. We define the functor
	$\omega_f:\left<f_*\mathcal{O}_Y\right>_{\otimes}\rightarrow \Vectf_k$ by
	$\omega_f(E)=H^0(\Strat(Y), f^*E)=(E\otimes f_*\mathcal{O}_Y)^{\nabla}$, and this
	functor is faithful and exact since $f^*E$ is trivial for all $E$. Finally, the
	 group scheme $G(\left<f_*\mathcal{O}_Y\right>_{\otimes},\omega)$ is the
	constant group scheme associated with $G$: Since $Y\times_X Y\cong Y\times_k G$, we see
	that $\omega_f(f_*\mathcal{O}_Y)$ is the right regular representation of $G$.

	Lastly, \ref{prop:universaltorsorF} follows immediately from \ref{prop:universaltorsorB}.
\end{proof}
\begin{Corollary}\label{cor:galoisClosure}\label{cor:FiniteEtaleIsFiniteStratified}
	If $f:Y\rightarrow X$ is a finite \'etale morphism, and $f':Y'\rightarrow X$ its
	Galois closure with Galois group $G$, then
	\[\mathcal{Y}:=\left<f_*\mathcal{O}_Y\right>_{\otimes}=\left<f'_{*}\mathcal{O}_{Y'}\right>_{\otimes}\subset
	\Strat(X),\]
	$f'=h_{\mathcal{Y},\omega_{f'}}$, and $f'_*\mathcal{O}_{Y'}$ is
	finite  with monodromy group the constant $k$-group scheme associated with $G$.
\end{Corollary}
\begin{proof}
	By Proposition \ref{prop:universaltorsor}, \ref{prop:universaltorsorA}, it follows
	that there is an inclusion $\mathcal{Y}\subset \left<f'_*\mathcal{O}_{Y'}\right>_{\otimes}$, so if
	$h_{f_*\mathcal{O}_Y,\omega_{f'}}:X_{f_*\mathcal{O}_Y,\omega_{f'}}\rightarrow X$ is the associated
	Galois \'etale morphism, then $h_{f_*\mathcal{O}_Y,\omega_{f'}}$ factors through $f$
	(because $h_{f_*\mathcal{O}_Y,\omega_{f'}}^*(f_*\mathcal{O}_Y)$ is trivial), and then
	by Proposition \ref{prop:universaltorsor},
	\ref{prop:universaltorsorF}, there is a morphism $Y'\rightarrow
	X_{f_*\mathcal{O}_Y,\omega_{f'}}$, such that the diagram
	\begin{equation*}
		\xymatrix{ 
		X_{f_*\mathcal{O}_Y,\omega_{f'}}\ar[dr]\ar[ddr]_--{h_{f_*\mathcal{O}_Y,\omega_{f'}}}&&
		Y'\ar[ll]\ar[dl]\ar[ddl]^{f'}\\
		&Y\ar[d]^--{f}\\
		& X
		 }
	 \end{equation*}
	 commutes. But since $f'$ is the Galois closure of $f$, it follows that
	 $Y'\rightarrow X_{f_*\mathcal{O}_Y,\omega_{f'}}$ is an isomorphism. This shows that
	 $G(f_*\mathcal{O}_Y,\omega_f')=G$.
\end{proof}
\begin{Corollary}\label{cor:trivialDmod}
	Let $f:Y\rightarrow X$ be a finite \'etale morphism of smooth, connected,
	separated $k$-varieties. Then $f$ is the trivial covering if and only if the
	stratified bundle $f_*\mathcal{O}_Y\in \Strat(X)$ is trivial.
\end{Corollary}
\begin{proof}
	Let $f':Y'\rightarrow X$ be the Galois closure of $f$. Then
	$G(f_*\mathcal{O}_Y,\omega)\cong G(f'_*\mathcal{O}_{Y'},\omega)\cong\Gal(Y'/X)_{k}$, so
	the claim follows.
\end{proof}
\begin{Definition}\label{defn:pvtorsor}
	Let $E\in \Strat(X)$ be a stratified bundle and
	$\omega:\left<E\right>_{\otimes}\rightarrow \Vectf_k$ a fiber functor. Write
	$h_{E,\omega}:X_{E,\omega}\rightarrow X$ for the smooth
	$G(\left<E\right>_{\otimes},\omega)$-torsor
	associated with $\left<E\right>_{\otimes}$ and $\omega$ in
	Proposition \ref{prop:universaltorsor}. We call $h_{E,\omega}$ the \emph{Picard-Vessiot torsor
	associated with $E$ and $\omega$}. 
\end{Definition}

\begin{Corollary}\label{prop:CharacterizationOfFiniteness}
	A stratified bundle $E$ on
	$X$ has finite monodromy if and only if there exists a finite \'etale covering $f:Y\rightarrow
	X$, such that $f^*E\in \Strat(Y)$ is trivial.
\end{Corollary}
\begin{proof}
	If $E$ has finite monodromy then an associated Picard-Vessiot torsor is finite
	\'etale and trivializes $E$. Conversely, if $h:Y\rightarrow X$ is finite \'etale
	and $h^*E$ trivial in $\Strat(Y)$, then $E\subset
	h_*h^*E=h_*\mathcal{O}_Y^{\rank E}$ in $\Strat(X)$. But $h_*\mathcal{O}_Y$ has
	finite monodromy, so $E$ has finite monodromy.
\end{proof}
\begin{Remark} A caution is in order: If $E$ is a stratified bundle with infinite
	monodromy, then	it is true that
	the
	$\mathscr{D}_{X/k}$-module $E\otimes_{\mathcal{O}_X} A_{E,\omega}$ is isomorphic to
	$A_{E,\omega}^{\rank E}$, but $h^*_{E,\omega}E\in \Strat(X_{E,\omega})$ is not
	trivial, because it also carries an action of the 
	differential operators relative to $X_{E,\omega}\rightarrow X$, which were trivial in the \'etale case. 
\end{Remark}

\section{Logarithmic Differential Operators}\label{sec:logdiffops}
We continue to denote by $k$ an algebraically closed field of positive characteristic $p$.
\begin{Definition}\label{defn:gpc}\leavevmode
	\begin{enumerate}[label=(\alph*), ref=(\alph*)]
		\item  Let $\overline{X}$ be a smooth, separated, finite type $k$-scheme, and $X\subset
			\overline{X}$ an open subscheme such that the boundary
			divisor $D_X:=\overline{X}\setminus X$ has strict
			normal crossings. We denote such a datum by $(X,\overline{X})$
			and call it a \emph{good partial compactification}.
		\item If $(X,\overline{X})$ is a good partial compactification, then
			$\mathscr{D}_{\overline{X}/k}(\log D_X)$
			denotes the sheaf of subalgebras of $\mathscr{D}_{\overline{X}/k}$ generated
			by those differential operators which locally fix all powers of
			the ideal of the boundary divisor.
	\end{enumerate}
\end{Definition}

\begin{Remark}\label{rem:log-diffops-in-coordinates}
	\begin{enumerate}[label=(\alph*), ref=(\alph*)]
		\item A good partial compactification $(X,\overline{X})$ gives rise to a
			log-scheme  over $\Spec k$
			with its trivial log-structure, in the sense of
			\cite{Kato/LogSchemes}. The sheaf 
			$\mathscr{D}_{\overline{X}/k}(\log D_X)$ is an
			invariant of this relative log-scheme, which can be constructed using an
			appropriate notion of thickenings in the category of fine
			saturated log-schemes. For details see
			e.g.~\cite[Ch.~2]{Kindler/thesis}.
		\item\label{rem:log-diffops-in-coordinates2} If $\overline{U}\subset \overline{X}$ is an open affine
			subset and
			$x_1,\ldots, x_n\in
			H^0(\overline{U},\mathcal{O}_{\overline{U}})$
			coordinates (see the notational conventions at the end
			of the introduction), then
			$\mathscr{D}_{\overline{U}/k}$ is generated by operators
			$\partial_{x_i}^{(m)}$, $i=1,\ldots, n$, $m\geq 0$, see
			\cite[Ch.~2]{BerthelotOgus/Crystalline}. Here
			$\partial^{(m)}_{x_i}$ can be thought of as 
			\[\frac{1}{m!}\frac{\partial^{m}}{\partial x_i^m},\]
			even though this does not make sense as written. The point is that
			evaluation of $\partial^{m}/\partial x_i^m$ on powers of $x_i$
			always gives coefficients which are divisible by sufficiently high
			powers of $p$. 
			
			Moreover, if the strict normal crossings divisor 
			$\overline{U}\setminus X$ is defined by $x_1\cdot\ldots\cdot
			x_r=0$, $r\leq n$, then $\mathscr{D}_{\overline{U}/k}(\log
			D_X\cap \overline{U})$ is generated by operators of the form
			\[x_i^m\partial^{(m)}_{x_i},\text{ for }i=1,\ldots, r, m\geq 0\]
			and \[\partial^{(m)}_{x_i}\text{ for } i>r, m\geq 0.\]
			For notational convenience we write
			$\delta_{x_i}^{(m)}:=x_i^m\partial^{(m)}_{x_i}$.
	\end{enumerate}
\end{Remark}

To proceed, we need to recall a few elementary facts about congruences for binomial coefficients:
\begin{Lemma}\label{lemma:binomials} Let $p$ be a prime number.
	\begin{enumerate}
		\item\label{lemma:lucas} \emph{Lucas' Theorem:} For
			$a_0,\ldots, a_n,b_0,\ldots, b_n$ integers in $[0,p-1]$,
			$a:=a_0+a_1p+\ldots+a_np^n$, $b:=b_0+b_1p+\ldots+b_np^n$ we have
			\[ {a \choose b}\equiv \prod_k {a_k\choose
			b_k}\mod p.\] 
\item\label{lemma:padicExpansionViaBinomials} If $\alpha\in \Z_p$, then
			\[\alpha = \sum_{n=0}^\infty \overline{{\alpha\choose p^n}}p^n,\]
			where $\overline{a}$ means the unique integer in $[0,\ldots,
			p-1]$ congruent to $a$.
		\item If $\alpha,\beta\in\Z_p$, $d\geq 0$, then
			\[{\alpha\beta\choose p^d} \equiv \sum_{\begin{subarray}{c}a+b=d\\a,b\geq
				0\end{subarray}}{\alpha\choose p^a}{\beta\choose p^b}\mod
				p.\]
	\end{enumerate}
\end{Lemma}
\begin{proof}
	\begin{enumerate}[label=(\alph*), ref=(\alph*)]
		\item This is easily proven by computing the coefficient
			of $x^b$ in
			\[\sum^a_{k=0}{a\choose k}x^k=(1+x)^a\equiv \prod_{k=0}^n
			(1+x^{p^k})^{a_k} \mod p.\]
	       \item This is a consequence of the continuity of $x\mapsto \binom{x}{k}$,
		       and \ref{lemma:lucas}.
		\item This follows directly from \ref{lemma:padicExpansionViaBinomials}.
	\end{enumerate}
\end{proof}
We have the similar functoriality results for
$\mathscr{D}_{\overline{X}/k}(\log D_X)$ as for $\mathscr{D}_{\overline{X}/k}$:
\begin{Proposition}\label{prop:logfunctoriality}\leavevmode
	Let $(X,\overline{X})$ and $(Y,\overline{Y})$ be good partial
	compactifications with boundary divisors $D_X$ and $D_Y$, and 
	$\bar{f}:\overline{Y}\rightarrow \overline{X}$ a morphism such that
	$\bar{f}(Y)\subset X$, i.e.~such that $\bar{f}$ induces a morphism of the
	associated log-schemes. Write $f:=\bar{f}|_Y$. Then the following statements are true:
	\begin{enumerate}
		\item\label{logfunct1} \[\bar{f}^*\mathscr{D}_{\overline{X}/k}(\log
			D_X):=\mathcal{O}_{\overline{Y}}\otimes_{\bar{f}^{-1}\mathcal{O}_{\overline{X}}}
			\bar{f}^{-1}\mathscr{D}_{\overline{X}/k}(\log D_X)\] is a
			$(\mathscr{D}_{\overline{Y}/k}(\log
			D_Y),\bar{f}^{-1}\mathscr{D}_{\overline{X}/k}(\log
			D_X))$-bialgebra.
		\item\label{logfunct2} There exists a canonical morphism
			\[\mathscr{D}_{\overline{Y}/k}(\log
			D_Y)\xrightarrow{\bar{f}^{\sharp}}
			\bar{f}^*\mathscr{D}_{\overline{X}/k}(\log D_X)\]
			fitting in the commutative diagram
			\[\xymatrix{
			\mathscr{D}_{\overline{Y}/k}(\log
			D_Y)\ar@{^{(}->}[d]\ar[r]^{\bar{f}^{\sharp}}
			&\bar{f}^*\mathscr{D}_{\overline{X}/k}(\log D_X)\ar@{^{(}->}[d]\\
			\mathscr{D}_{\overline{Y}/k}\ar[r]&\bar{f}^*\mathscr{D}_{\overline{X}/k}}\]
			where the lower horizontal morphism is the classical
			one arising from the functoriality of the diagonal morphism and
			its thickenings.
			\end{enumerate}

	Now assume that $\bar{f}$ is finite, and $f$ \'etale. Then $\bar{f}$ is faithfully
	flat. Moreover,
	\begin{enumerate}[resume]
		\item\label{logfunct3} $\bar{f}^{\sharp}$ is an isomorphism if $\bar{f}$ is tamely ramified
			with respect to the strict normal crossings divisor
			$D_X$.
	\end{enumerate}
\end{Proposition}
\begin{proof}
	Everything follows fairly easily from the general point of view of logarithmic structures
	of \cite{Kato/LogSchemes}, since $\bar{f}$ being tamely ramified implies that the
	induced morphism of log-schemes is log-\'etale. 

	For the sake of self-containedness, we give an explicit proof for the case that
	$\bar{f}$ is finite and $f$ \'etale, which is the only case needed in the sequel.
	In this case
	the proof is essentially a question about finite extensions of discrete valuation
	rings, by localizing at the generic points of the components of the boundary
	divisor. Let $A\hookrightarrow B$ be such an extension, $x\in A$ and $y\in B$
	uniformizers. Then $x=uy^e$ for some $e\geq 1$ and $u\in B^\times$. We know
	that $K(B)\otimes \mathscr{D}_{B/k}(\log\;(y))\xrightarrow{\cong} K(B)\otimes_A
	\mathscr{D}_{A/k}(\log\;(x))$ is an isomorphism. Statement \ref{logfunct1} is clear,
	and for \ref{logfunct2} we need to show that the above isomorphism maps
	$\mathscr{D}_{B/k}(\log\;(y))$ to $B\otimes_{A}\mathscr{D}_{A/k}(\log\;(x))$. We claim that
	\begin{equation}\label{extformula}
		\delta_y^{(p^m)}=\sum_{\begin{subarray}{c}d+c=m\\c,d\geq 0\end{subarray}}\binom{e}{p^c}\delta_x^{(p^d)} + y(B\otimes_A
		\mathscr{D}_{A/k}(\log\;(x))\end{equation}
		where $\delta_{y}^{(p^m)}=y^{p^m}\partial_{y}^{(p^m)}$ as usual.
		Clearly \eqref{extformula} implies \ref{logfunct2}.

	Let us now prove that \eqref{extformula} holds. We compute:
	\begin{align}
		\delta^{(p^m)}_y(x^r)&=\delta^{(p^m)}_y(u^ry^{er})\nonumber\\
		&=\sum_{\begin{subarray}{c}a+b=p^m\\a,b\geq
			0\end{subarray}}\delta^{(a)}_y(u^r)\delta^{(b)}_{y}(y^{er})\nonumber\\
			&=\binom{er}{p^m}x^r+\sum_{\begin{subarray}{c}a+b=p^m\\a>0,b\geq
				0\end{subarray}}
				\delta^{(a)}_{y}(u^r)\binom{er}{b}y^{er}\nonumber\\
				&=\binom{er}{p^m}x^r+x^r\underbrace{\sum_{\begin{subarray}{c}a+b=p^m\\a>0,b\geq
					0\end{subarray}}\binom{er}{b}\frac{\delta^{(a)}_y(u^r)}{u^r}}_{\in
					(y)}\label{functformula}
				\end{align}
 Note that
	$\binom{er}{p^m}x^r=\sum_{c+d=m}\binom{e}{p^c}\delta^{(p^d)}_x(x^r)$ by
	Lemma \ref{lemma:binomials},
	so \eqref{functformula} shows that $\delta^{(p^m)}_y-\sum_{c+d=m}\binom{e}{p^c}\delta^{(p^d)}_x\in
	y(B\otimes \mathscr{D}_{A/k}(\log\;(x))$ as claimed.

	For \ref{logfunct3} assume that $A\hookrightarrow B$ is tamely ramified. It suffices to show that $\delta^{(p^m)}_{x}$ is in the image
	of $\bar{f}^{\sharp}:\mathscr{D}_{B/k}(\log\;(y))\rightarrow
	\bar{f}^*\mathscr{D}_{A/k}(\log\;(x))$ for every $m\geq 0$, because $\bar{f}^{\sharp}$ is injective
	by the separability of the residue extensions. Consider the completions of $A$ and
	$B$: $\widehat{A}\hookrightarrow\widehat{B}$. Replacing $\widehat{B}$ 
	by an \'etale extension does not change differential operators, so we may assume that $u=v^e$ in
	$\widehat{B}$. Indeed, by Hensel's Lemma, $u$ has an $e$-th root in $\widehat{B}$,
	if and only if it has an $e$-th root modulo $y$, and since $e$ is prime to $p$ by
	assumption, the extension of the residue fields obtained by adjoining an $e$-th
	root is separable. Replacing $y$ by $vy$, we may assume that $x=y^e$. Then
	\eqref{functformula} shows
	\[\delta^{(p^m)}_y=\sum_{\begin{subarray}{c}c+d=m\\c,d\geq 0
	\end{subarray}}\binom{e}{p^c}\delta^{(p^d)}_x\]
	In particular, $\delta^{(1)}_y=e\delta^{(1)}_x$, so $\delta^{(1)}_x$ is in the
	image of $\bar{f}^{\sharp}$. We proceed inductively: 
	\[\delta^{(p^m)}_{y}=e\delta_x^{(p^m)}+\underbrace{\sum_{\begin{subarray}{c}c+d=m\\c>1, d\geq 0
	\end{subarray}}\binom{e}{p^c}\delta^{(p^d)}_x}_{\in\image{\bar{f}^\sharp}}\]
	which completes the proof.
\end{proof}

\begin{Corollary}\label{cor:logpullback}
	We continue to use the notations from Proposition \ref{prop:logfunctoriality}.
	If $E$ is a $\mathscr{D}_{\overline{X}/k}(\log D_X)$-module, then
	$\bar{f}^*E$ is a $\mathscr{D}_{\overline{Y}/k}(\log D_Y)$-module.

	If $\bar{f}$ is finite \'etale and tamely ramified with respect to $\overline{X}\setminus X$,
	and $F$ a $\mathscr{D}_{\overline{Y}/k}(\log D_Y)$-module, then
	$\bar{f}_* {F}$ is a $\mathscr{D}_{\overline{X}/k}(\log D_X)$-module.
\end{Corollary}

\section{$(X,\overline{X})$-Regular Singular Stratified Bundles}\label{sec:rs}
\begin{Definition}\label{defn:rs}
		If $(X,\overline{X})$ is a good partial compactification, then a
			stratified bundle $E$
			is called
			$(X,\overline{X})$-\emph{regular singular} if it extends to an
			$\mathcal{O}_{\overline{X}}$-torsion-free, $\mathcal{O}_{\overline{X}}$-coherent
			$\mathscr{D}_{\overline{X}/k}(\log D_X)$-module $\overline{E}$ on
			$\overline{X}$.

	We define $\Strat^{\rs}( (X,\overline{X}))$ to be the full
			subcategory of $\Strat(X)$ consisting of
			$(X,\overline{X})$-regular singular bundles.
\end{Definition}
\begin{Remark}
	The notion of an $(X,\overline{X})$-regular singular stratified bundle with
	$(X,\overline{X})$ an actual good compactification (i.e.~$\overline{X}$ proper)
	was studied in \cite{Gieseker/FlatBundles}. Many of his arguments do
	not use the properness of $\overline{X}$ and carry over to
	our setup.
\end{Remark}

The following proposition shows that to check that a stratified bundle is 
$(X,\overline{X})$-regular singular, we may always assume that
$\overline{X}\setminus X$ is a smooth divisor, and we may shrink $\overline{X}$ around the
generic points of $\overline{X}\setminus X$.
\begin{Proposition}\label{prop:locality}
	Let $(X,\overline{X})$ be a good partial compactification, $\eta_1,\ldots, \eta_r$
	the generic points of $\overline{X}\setminus X$, and $E$ a stratified
	bundle on $X$. Then $E$ is $(X,\overline{X})$-regular singular, if and only if there are open neighborhoods
	$\overline{U}_i$ of $\eta_i$, $i=1,\ldots, r$, such that $E|_{\overline{U}_i\cap
	X}$ is $(\overline{U}_i\cap
	X,\overline{U}_i)$-regular singular.
\end{Proposition}
\begin{proof}
	Only the ``if'' direction is interesting. Given open neighborhoods $\overline{U}_i$
	of $\eta_i$, $i=1,\ldots, r$, as in the proposition, we may
	assume, by shrinking the $\overline{U}_i$ if necessary, that $\eta_j\not\in
	\overline{U}_i$ if $i\neq j$ and then that $X\subset \overline{U}_i$.
	Hence, if $\overline{E}_i$ is an
	$\mathcal{O}_{\overline{X}}$-torsion-free $\mathcal{O}_{\overline{U}_i}$-coherent extension of $E$
	to $\overline{U}_i$ as 
	$\mathscr{D}_{\overline{U}_i/k}(\log D_X\cap \overline{U}_i)$-module, then we
	can glue to obtain an extension $\overline{E}'$ on $\bigcup_i \overline{U}_i$.  Write $\overline{U}:=\bigcup_i \overline{U_i}$, then
	$\overline{X}\setminus \overline{U}$ has codimension $\geq 2$ in
	$\overline{X}$. If $j:\overline{U}\hookrightarrow \overline{X}$ be the open
	immersion, then $j_*E'$ is an
	$\mathcal{O}_{\overline{X}}$-torsion-free, $\mathcal{O}_{\overline{X}}$-coherent
	$\mathscr{D}_{\overline{X}/k}(\log D_X)$-module.
\end{proof}

We also have the notion of a pullback of $(X,\overline{X})$-regular singular stratified
bundles:
\begin{Proposition}\label{prop:pullback}
	Let $(X,\overline{X})$, $(Y,\overline{Y})$ be good partial compactifications and
	$\bar{f}:\overline{X}\rightarrow \overline{Y}$ a $k$-morphism, such that
	$\bar{f}(X)\subset Y$. Then for every $(X,\overline{X})$-regular singular
	stratified bundle on $X$, $(\bar{f}|_Y)^*E$ is $(Y,\overline{Y})$-regular singular.
\end{Proposition}
\begin{proof}This is a direct consequence of Corollary \ref{cor:logpullback}.\end{proof}

\begin{Proposition}\label{prop:subbundlesAreRS}\label{prop:rs-tannakian}
	Let $(X,\overline{X})$ be a good partial compactification and $E, E'$ be
	$(X,\overline{X})$-regular singular bundles.
	Then the following stratified bundles are also 	$(X,\overline{X})$-regular singular:
	\begin{enumerate}
		\item\label{rs-tannakian1}  Every substratified bundle $F\subset E$.
		\item\label{rs-tannakian2} Every quotient-stratified bundle $E/F$ of $E$.
		\item\label{rs-tannakian3} $E\otimes_{\mathcal{O}_X} E'$.
		\item\label{rs-tannakian4} $\mathcal{H}om_{\mathcal{O}_X}(E,E')$.
	\end{enumerate}
	It follows that $\Strat^{\rs}((X,\overline{X}))$ is a sub-Tannakian subcategory of
	$\Strat(X)$, and, if $\iota:\Strat^{\rs}( (X,\overline{X}) )\hookrightarrow \Strat(X)$
	denotes the inclusion functor, then for a $(X,\overline{X})$-regular singular
	bundle $E$, restriction of $\iota$ gives an equivalence $\left<E\right>_{\otimes}\rightarrow
	\left<\iota(E)\right>_{\otimes}$.
\end{Proposition}
\begin{proof}
	Again we write $D_X:=\overline{X}\setminus X$, $j:X\hookrightarrow \overline{X}$.
	For \ref{rs-tannakian1}, let $\overline{E}$ be an $\mathcal{O}_{\overline{X}}$-torsion-free $\mathcal{O}_{\overline{X}}$-coherent
	$\mathscr{D}_{\overline{X}/k}(\log D_X)$-module extending $E$. Then $j_*F$ and $\overline{E}$ are both
	$\mathscr{D}_{\overline{X}/k}(\log D_X)$-submodules of $j_*E$; let $\overline{F}$ be their
	intersection. Then $\overline{F}$ and
	$\overline{E}/\overline{F}$ are $\mathcal{O}_{\overline{X}}$-coherent,
	$\mathcal{O}_{\overline{X}}$-torsion free
	$\mathscr{D}_{\overline{X}/k}(\log D_X)$-modules extending
	$F$, resp.~$E/F$. 

	If $\overline{E},\overline{E'}$ are $\mathcal{O}_{\overline{X}}$-coherent
	$\mathscr{D}_{\overline{X}/k}(\log D_X)$-modules extending $E$ and $E'$, then \ref{rs-tannakian3} and \ref{rs-tannakian4} follow from the fact that
	$\HHom_{\mathcal{O}_{\overline{X}}}(\overline{E},\overline{E}')$ and
	$\overline{E}\otimes_{\mathcal{O}_{\overline{X}}} \overline{E}'$ carry natural
	$\mathscr{D}_{\overline{X}/k}(\log D_X)$-actions extending those coming from $E$ and
	$E'$.
\end{proof}

\begin{Proposition}[{\cite[Lemma 3.8]{Gieseker/FlatBundles}}]\label{prop:globaldecomp}
	Let $(X,\overline{X})$ be a good partial compactification such that
	$D_X:=\overline{X}\setminus X$ is a smooth divisor, and $i:D_X\hookrightarrow
	\overline{X}$ the closed immersion ($D_X$ reduced). Write
	\[\overline{\mathscr{D}}:=\ker\left(\mathscr{D}_{\overline{X}/k}(\log D_X)\rightarrow
	i_*i^*\mathscr{D}_{\overline{X}/k}\right).\]
	If
	$\overline{E}$ is a locally free
	$\mathcal{O}_{\overline{X}}$-coherent
	$\mathscr{D}_{\overline{X}/k}(\log D_X)$-module, then
	$\overline{\mathscr{D}}$ acts $\mathcal{O}_{D_X}$-linearly on
	$\overline{E}|_{D_X}$, and there exists a decomposition
	\begin{equation} \label{eqn:globaldecomp}
		\overline{E}|_{D_X}=\bigoplus_{\alpha\in\Z_p}
		F_{\alpha},
	\end{equation}
	such that $\theta\in \overline{\mathscr{D}}$ acts on $F_{\alpha}$ via
	$\theta(s)=\bar{\alpha} s$, where $\bar{\alpha}$ is defined as follows: If
	$y$ is a local defining equation for $D_X$, $p^N> \ord \theta$ and $\beta\in
	\N$, such that $\beta \equiv \alpha \mod p^N$, then
	\begin{equation}\label{eqn:coordinatefreedescription}
			\theta(y^\beta)=\bar{\alpha}y^{\beta}\mod y^{\beta+1}.
	\end{equation}
\end{Proposition}
\begin{Remark}
	\begin{enumerate}[label=(\alph*), ref=(\alph*)]
		\item Let us unravel the definition of the $F_{\alpha}$ after the choice of local
	coordinates $x_1,\ldots, x_n$ as in
	Remark \ref{rem:log-diffops-in-coordinates}
	\ref{rem:log-diffops-in-coordinates2}, such that $D_X$ is locally cut out by, say, $x_1$: A section $s$ of $E$
	has the property that $s+x_1E \in F_{\alpha}$ if and only if
	$\delta_{x_1}^{(m)}(s)=\binom{\alpha}{m}s+x_1E$.
	The more complicated description \eqref{eqn:coordinatefreedescription} from the
	proposition shows that the decomposition \eqref{eqn:globaldecomp}
	does not depend on the choice of coordinates and exists on all of $D_X$.
\item The existence of the decomposition \eqref{eqn:globaldecomp} is the reason for a
	rather striking difference of our setup from the characteristic $0$ situation:
	Proposition \ref{prop:globaldecomp}	has as a consequence (see Proposition \ref{cor:essentialimage}) that in our setup
	there are no objects analogous to regular singular flat connections with
	nilpotent but nontrivial residues.\end{enumerate}
\end{Remark}
\begin{Definition}
	Let $\overline{E}$ be an $\mathcal{O}_{\overline{X}}$-locally free
	$\mathscr{D}_{\overline{X}/k}(\log D_X)$-module with
	$D_X:=\overline{X}\setminus X$ smooth.
	The elements $\alpha\in \Z_p$ such that $F_\alpha\neq 0$ in the decomposition
	\eqref{eqn:globaldecomp}, are called \emph{exponents of
		$\overline{E}$} along $D_X$. If $D_X$
	is not smooth, but $D_X=\bigcup_{i=1}^r D_i$, $D_i$  smooth divisors, then
	\emph{the exponents of $\overline{E}$ along $D_i$} are defined by restricting
	$\overline{E}$ to an open set $U\subset \overline{X}$ intersecting $D_i$, but not
	$D_j$, for $j\neq i$. Finally, if $\overline{E}$ is
	$\mathcal{O}_{\overline{X}}$-torsion-free, but not locally free, then the
	exponents are defined by restricting to an open subset $\overline{U}\subset
	\overline{X}$ such that $\codim_{\overline{X}}\overline{X}\setminus
	\overline{U}\geq 2$ and such that $E|_{\overline{U}}$ is locally free.

	Propostion \ref{prop:globaldecomp} shows that this definition does not depend on
	the choices made.
\end{Definition}

We state another analogy to the characteristic
$0$ situation (see \cite{Levelt}):

\begin{Theorem}\label{thm:reflexive}Let $(X,\overline{X})$ be a good partial compactification, and $E$ an $\mathcal{O}_{\overline{X}}$-torsion-free,
	$\mathcal{O}_{\overline{X}}$-coherent
	$\mathscr{D}_{\overline{X}/k}(\log D_X)$-module. If the exponents of $E$ do not differ by integers, then $E$ is locally free if and
	only if it is reflexive. 
\end{Theorem}
\begin{proof}
	This is a direct consequence of \cite[Thm.~3.5]{Gieseker/FlatBundles}.
\end{proof}

\begin{Remark}
	We will see later on (Theorem \ref{thm:tauExtension}) that an $(X,\overline{X})$-regular singular
	stratified bundle always has an extension to an
	$\mathcal{O}_{\overline{X}}$-locally free
	$\mathscr{D}_{\overline{X}/k}(\log D_X)$-module.
\end{Remark}

\begin{Proposition}\label{prop:pullbackmultipliesexponents}
	Let $(Y,\overline{Y})$ and $(X,\overline{X})$ be good partial
	compactifications with boundary divisors $D_Y$ and $D_X$,
	and $\bar{f}:\overline{Y}\rightarrow \overline{X}$ a finite morphism, such that
	$\bar{f}(Y)\subset X$, and such that $f:=\bar{f}|_Y$ is \'etale. Let
	$D'_Y$ be a
	component of $D_Y$ mapping to the component $D'_X$ of
	$D_X$.  If
	$E$ an $(X,\overline{X})$-regular singular stratified bundle, then the exponents
	of the $(Y,\overline{Y})$-regular singular bundle $f^*E$ along $D'_Y$ are the
	exponents of $E$ along $D'_X$ multiplied by the ramification index of
	$D'_Y$ over $D'_X$.
\end{Proposition}
\begin{proof}
	This is again a question about discrete valuation rings. Let $A\hookrightarrow B$
	be a finite extension of discrete valuation rings, such that the extension of
	fraction fields $K(A)\hookrightarrow K(B)$ is separable. Let $x$ be a uniformizer
	for $A$ and $y$ a uniformizer for $B$, $x=uy^e$ with $u\in B^\times$.
	We use the computation from Proposition \ref{prop:logfunctoriality}.
	Let $E$ be an $A$-module with $\mathscr{D}_{A/k}(\log\;(x))$-action. If $a\in
	E$ is such that $\delta^{(p^m)}_x(a)=\binom{\alpha}{p^m}a+xE$ for some
	$\alpha\in \Z_p$, then \eqref{extformula} shows that
	\[\delta^{(p^m)}_y(a\otimes 1)=\sum_{\begin{subarray}{c}c+d=m\\c,d\geq
		0\end{subarray}}\binom{e}{p^c}\binom{\alpha}{p^d}a\otimes 1+y(E\otimes
	B)=\binom{e\alpha}{p^m}a\otimes 1+y(E\otimes B)\]
	which proves the proposition.
\end{proof}

\begin{Proposition}\label{prop:exponentsModZ}
	Let
	$(X,\overline{X})$ be a good partial compactification such that
	$D_X:=\overline{X}\setminus X$ is smooth. Let $E$ be an $(X,\overline{X})$-regular
	singular stratified bundle, and let $\overline{E}$ and 
	$\overline{E}'$ be two locally free $\mathcal{O}_{\overline{X}}$-coherent
	$\mathscr{D}_{\overline{X}/k}(\log D_X)$-modules, such that
	$\overline{E}'|_X=\overline{E}|_X=E$ as $\mathscr{D}_{X/k}$-modules. Write
	\[ \overline{E}|_{D_X}=\bigoplus_{\alpha\in\Z_p} F_{\alpha} \text{ and }
		\overline{E}'|_{D_X} = \bigoplus_{\alpha\in\Z_p} F'_{\alpha},\]
	as in Proposition \ref{prop:globaldecomp}, and  define $\Exp(\overline{E}):=\left\{ \alpha\in\Z_p| F_{\alpha}\neq 0 \right\}$,
	$\Exp(\overline{E}'):=\left\{\alpha\in\Z_p| F'_{\alpha}\neq 0\right\}$.  Then the
	images of $\Exp(\overline{E})$ and
	$\Exp(\overline{E}')$ in $\Z_p/\Z$ are identical.
\end{Proposition}
\begin{proof}
	Let 
	$\eta$ be the generic point of $\overline{X}\setminus X$. To prove the
	proposition, we may shrink $\overline{X}$ around $\eta$, so that
	we can assume that $\overline{X}=\Spec A$ is affine,
	$\overline{E},\overline{E}'$ are free, and that there are coordinates
	$x_1,\ldots, x_n\in
	H^0(\overline{X},\mathcal{O}_{\overline{X}})$ such that $D_X=(x_1)$. Write $j:X\hookrightarrow
	\overline{X}$. Then $\overline{E}\cap \overline{E}', \overline{E}, \overline{E}'\subset j_*E$ are
	$\mathscr{D}_{\overline{X}/k}(\log D_X)$-submodules, and replacing
	$\overline{E}$ by $\overline{E}\cap \overline{E}'$,  we may assume that $\overline{E}\subset \overline{E}'$ is a
	$\mathscr{D}_{\overline{X}/k}(\log D_X)$-submodule.
	
	We may  now consider the situation over $\mathcal{O}_{\overline{X},\eta}$, which is a discrete valuation
	ring with uniformizer $x_1$.	
	For some $n\in \N$ we have $x_1^n\overline{E}'_{\eta}\subset
	\overline{E}_{\eta}\subset \overline{E}'_{\eta}$, and it is
	not difficult to compute that
	$\Exp(x_1^n\overline{E}'_{\eta})=\{\alpha+n|\alpha\in
	\Exp(\overline{E}'_{\eta})\}$.
	Thus it suffices
	to show that if $\alpha\in \Z_p$ is an exponent of
	$\overline{E}_{\eta}$, then $\alpha+N_{\alpha}$ is an exponent of
	$\overline{E}'_{\eta}$ for some $N_{\alpha}\in \Z$.

	Let $\alpha_1,\ldots, \alpha_r$ be the exponents of $\overline{E}_{\eta}$.
	If $e\in \overline{E}_{\eta}\setminus
	x_1\overline{E}'_{\eta}$ is an element such that
	$\delta^{(m)}_{x_1}(e)=\binom{\alpha_1}{m}e+x_1\overline{E}_{\eta}$,
	then $\alpha_1$ also is an exponent of
	$\overline{E}'_{\eta}$.
	If there is no such $e$, let $e_1,\ldots, e_r$ be a lift of a basis of
	$\overline{E}_{\eta}/x_1$, such that
	$\delta_{x_1}^{(m)}(e_i)=\binom{\alpha_i}{m}e_i+x_1\overline{E}_{\eta}$
	for all $m\geq 0$,
	and define $\overline{E}_{-1}$ as the submodule of
	$\overline{E}'_{\eta}$ spanned by $x_1^{-1}e_1,e_2\ldots, e_r$. Note that
	$\overline{E}_{\eta}\subsetneq\overline{E}_{-1}$. It is
	readily checked that $\overline{E}_{-1}$ is stable under the
	$\delta^{(m)}_{x_1}$, and we show that $\alpha_1-1$ is an exponent of
	$\overline{E}_{-1}$.        
	Since $\mathcal{O}_{\overline{X},\eta}/(x_1)$ is a field, and since
	$e_i\not\in x_1\overline{E}_{-1}$ for $i>1$, one finds
	$t\in \overline{E}_{-1}$, such that
	$t,e_2,\ldots, e_r$ is a lift of a	basis of $\overline{E}_{-1}/x_1\overline{E}_{-1}$, and such that
	$\delta_{x_1}^{(m)}(t)=\binom{\beta}{m}t+x_1\overline{E}_{-1}$,
	for some $\beta\in \Z_p$.
	We compute:
	\[\delta^{(m)}_{x_1}(x_1^{-1}e_1)=\binom{\alpha_1-1}{m}x_1^{-1}e_1+f+x_1\overline{E}_{-1}\]
	for some $f\in \left<e_{2},\ldots,e_r\right>$.
	On the other hand, writing 
	$x_1^{-1}e_1=\lambda_1 t+\sum_{j>1}\lambda_j
	e_j$ with $\lambda_i \in \mathcal{O}_{\overline{X},\eta}$, we get
	\[
		\delta^{(m)}_{x_1}(x_1^{-1}e_1)=
		\binom{\beta}{m}\lambda_1 t +
		\sum_{j>1}\binom{\alpha_j}{m}\lambda_{j}e_j+x_1\overline{E}_{-1}.
	\]
	Since  $x_1^{-1}e_{1}\not\in \left<e_2,\ldots,
	e_r\right>$, we have $\lambda_1\not\in (x_1)$. Hence, comparing
	coefficients shows that $\beta=\alpha_1-1$,
	which shows that $\alpha_1-1$ is an exponent of
	$\overline{E}_{-1}$.

	If there is $e\in
	\overline{E}_{-1}\setminus x_1\overline{E}'_{\eta}$ such
	that
	$\delta_{x_1}^{(m)}(e)=\binom{\alpha_1-1}{m}e+x_1\overline{E}_{-1}$
	then $\alpha_1-1$ is an exponent of $\overline{E}'_{\eta}$. Otherwise we construct
	$\overline{E}_{-2}:=(\overline{E}_{-1})_{-1}\supsetneq \overline{E}_{-1}$ contained in
	$\overline{E}'_{\eta}$ as above, with exponent $\alpha_1-2$. Since
	$\overline{E}'_{\eta}/\overline{E}_{\eta}$ has finite length,
	this process has to terminate, so there is some $n$ such that
	$\alpha_1-n$ is an exponent of $\overline{E}'_{\eta}$.

		\end{proof}
\begin{Definition}
	Let $(X,\overline{X})$ be a good partial compactification, such that
	$D_X:=\overline{X}\setminus X=\bigcup_{i=1}^r D_i$, with $D_i$ smooth divisors. If
	$E$ is an
	$(X,\overline{X})$-regular singular bundle on $X$, let $\overline{E}$ be an
	$\mathcal{O}_{\overline{X}}$-torsion-free $\mathcal{O}_{\overline{X}}$-coherent $\mathscr{D}_{\overline{X}/k}(\log
	D_X)$-module extending $E$.  Then write $\Exp_i(E)$ for the image of
	the set of exponents of $\overline{E}$ along $D_i$ in $\Z_p/\Z$. The set $\Exp_i(E)$ is independent of the choice of $\overline{E}$ and
	it is called the \emph{set of exponents of $E$ along $D_i$}. Finally, write
	$\Exp_{(X,\overline{X})}(E)=\bigcup_i \Exp_i(E)$.
\end{Definition}
\begin{Remark}
	We emphasize that by definition the exponents of an $(X,\overline{X})$-regular singular bundle
	are elements of $\Z_p/\Z$, while the exponents of an $\mathcal{O}_{\overline{X}}$-locally free
	$\mathscr{D}_{\overline{X}/k}(\log D_X)$-module are elements
	of $\Z_p$.
\end{Remark}
\section{$\tau$-Extensions of $(X,\overline{X})$-Regular Singular Stratified Bundles}\label{sec:tauextension}
In this section, we study in which ways a given $(X,\overline{X})$-regular
singular bundle $E$ can extend to a $\mathcal{O}_{\overline{X}}$-locally free
$\mathscr{D}_{\overline{X}/k}(\log D_X)$-module.
\begin{Definition}
	Let $(X,\overline{X})$ be a good partial compactification,
	$D_X$ the boundary divisor, and $\tau:\Z_p/\Z\rightarrow
	\Z_p$ a set-theoretical section of the projection $\Z_p\rightarrow \Z_p/\Z$. If $E$ is an
	$(X,\overline{X})$-regular singular bundle, then a \emph{$\tau$-extension of $E$}
	is a finite rank, $\mathcal{O}_{\overline{X}}$-locally free
	$\mathscr{D}_{\overline{X}/k}(\log D_X)$-module $\overline{E}^{\tau}$ such that the
	exponents of $\overline{E}^{\tau}$ lie in the image of $\tau$.
\end{Definition}

\begin{Theorem}\label{thm:tauExtension}
	If $(X,\overline{X})$ is a good partial compactification,
	$D_X:=\overline{X}\setminus X$, $E$ an $(X,\overline{X})$-regular singular
	bundle on $X$, and $\tau:\Z_p/\Z\rightarrow \Z_p$ a section of the projection, then
	a $\tau$-extension of $E$ exists and is unique up to isomorphisms which restrict
	to the identity $E\rightarrow E$ on $X$.
\end{Theorem}
\begin{proof}
	This proof is an extension of the method of \cite[Lemma
	3.10]{Gieseker/FlatBundles}.

	From Proposition \ref{prop:locality} and Theorem \ref{thm:reflexive} it follows easily that we may assume
	without loss of generality that $D_X$ is a smooth divisor with generic
	point $\eta$, and  to prove the proposition we may shrink
	$\overline{X}$ around $\eta$. Hence, we assume that
	$\overline{X}$ is affine with coordinates $x_1,\ldots, x_n$, such that
	$D_X=(x_1)$, that $E$ is free of rank $r$, and that there exists an
	$\mathcal{O}_{\overline{X}}$-free $\mathscr{D}_{\overline{X}/k}(\log
	D_X)$-extension $\overline{E}$ of $E$.

 Let $\Exp(\overline{E})=\left\{
	\alpha_1,\ldots,\alpha_r \right\}$ be the set of exponents of $\overline{E}$ along
	$D_X$. 
	To prove the existence of a $\tau$-extension, we 
	proceed in two steps:
	\begin{enumerate}[label=(\alph*), ref=(\alph*)]
		\item\label{stepa} If $a\in \Z_{+}$ then			there exists an $\mathcal{O}_{\overline{X}}$-free
			$\mathscr{D}_{\overline{X}/k}(\log D_X)$-module
			$\overline{E}^{(-a)}$ with exponents
			\[ \Exp\left(\overline{E}^{(-a)}\right)=\left\{ \alpha_1-a,\ldots, \alpha_r-a
			\right\},\]
			such that $\overline{E}^{(-a)}|_X=E$.
			Indeed, we can take
			$\overline{E}^{(-a)}:=\overline{E}(aD_X)=\overline{E}\otimes_{\mathcal{O}_{\overline{X}}}
			\mathcal{O}_{\overline{X}}(aD_X)$, because
			$\mathcal{O}_{\overline{X}}(aD_X)$ is defined by $x_1^{-a}$. 
		       \item\label{stepb} From $\overline{E}$ we construct an
			       $\mathcal{O}_{\overline{X}}$-free
			       $\mathscr{D}_{\overline{X}/k}(\log D_X)$-module
			       $\overline{E}_i$ extending $E$, such that 
			       \[\Exp(\overline{E}_i)=\left\{ \alpha_j
				       |\alpha_j \neq \alpha_i
			       \right\}\cup \left\{ \alpha_i+1 \right\}.\]
	\end{enumerate}
	Applying step \ref{stepa} for an appropriate $a\in\Z_+$, and then step
	\ref{stepb}
	repeatedly for various $i$, we obtain a $\tau$-extension.

	We construct $\overline{E}_i$ as in step \ref{stepb} for $i=1$.
	Assume that $\alpha_1=\ldots=\alpha_{\ell}$ and $\alpha_j\neq
	\alpha_1$ for $j>\ell$.
	After perhaps shrinking $\overline{X}$ around $\eta$ there exists a
	lift $e_1,\ldots, e_r\in \overline{E}$ of a basis of
	$\overline{E}/x_1\overline{E}$, such that
	$\delta_{x_1}^{(m)}(e_i)=\binom{\alpha_i}{m}e_i+x_1\overline{E}$ for
	all $m\geq 0$.
	If $j:X\hookrightarrow \overline{X}$ denotes the inclusion, define $\overline{E}_{1}$ to be the
	$\mathcal{O}_{\overline{X}}$-submodule of $j_*E$ spanned by
	$x_1e_1,\ldots, x_1e_{\ell},e_{\ell+1},\ldots, e_r$. Then
	$\left(\overline{E}_1\right)|_X=\overline{E}|_X$. Note that
	$\delta^{(m)}_{x_1}(e_i)\in \overline{E}_1$ for $i>\ell$. For $i\leq
	\ell$ we compute:
	\[\delta_{x_1}^{(m)}(x_1e_i)=\binom{\alpha_1+1}{m}x_1e_i+x_1^2
		\overline{E}\]
	As $x_1^2\overline{E}\subset x_1\overline{E}_1$, we see that
	$\delta_{x_1}^{(m)}(x_1e_i)\in \overline{E}_{1}$, and hence
	$\overline{E}_1$ is stable under the $\delta^{(m)}_{x_1}$.
	Moreover, since
	$x_1e_i\not\in x_1\overline{E}_1$, it follows that
	that
	$\delta_{x_1}^{(m)}(x_1e_i)=\binom{\alpha_1+1}{m}x_1e_i\not\equiv
	0 \!\mod x_1\overline{E}_1$, so  $\alpha_1+1$ is an exponent of
	$\overline{E}_1$.

	To compute the other exponents, note that for every $i>\ell$ there
	exists $f_i\in \left<x_1e_1,\ldots, x_1e_{\ell}\right>$, such that
	\begin{equation}\label{errorterm}\delta^{(m)}_{x_1}(e_i)=\binom{\alpha_i}{m}e_i+f_i+x_{1}\overline{E}_1.\end{equation}
	If $g_1,\ldots, g_r$ is a lift of a basis of $\overline{E}_{1}/x_1$
	such that
	$\delta^{(m)}_{x_1}(g_i)=\binom{\beta_i}{m}g_i+x_1\overline{E}_1$, for
	some $\beta_i\in \Z_p$, and
	$\beta_1=\ldots=\beta_{h}=\alpha_1+1$ for $\ell\leq h\leq r$, then for
	$i>\ell$ we write $e_i=\sum_{j=1}^r\lambda_jg_j$, $\lambda_j\in
	\mathcal{O}_{\overline{X}}$. There are two cases: If $\lambda_j\in (x_1)$ for all $j>h$, then $\alpha_i=\alpha_1+1$.
	Otherwise, if $j_0>h$ is such that $\lambda_{j_0}\not\in(x_1)$, we
	compute
	\[\delta^{(m)}_{x_1}(e_i)=\binom{\alpha_1+1}{m}\sum_{j\leq h}\lambda_j
		g_j + \sum_{j>h}\binom{\beta_j}{m}\lambda_j
		g_j+x_1\overline{E}_1,\]
	and compare coefficients with \eqref{errorterm}. 
	 It follows that $\beta_{j_0}=\alpha_i$, since $f_i\in \left<g_1,\ldots, g_h\right>$.
	This finishes the proof of the existence of a $\tau$-extension.
	Its unicity follows from the following lemma:
	\begin{Lemma}\label{lemma:localuniquenessoftauextensions}
		Let $\overline{X}=\Spec A$ be an affine $k$-variety and
		$x_1,\ldots, x_n\in A$ coordinates such that $D_X=(x_1)$. Let  $X$ be $\overline{X}\setminus
		D_X=\Spec A[x_1^{-1}]$ and $\tau:\Z_p/\Z \rightarrow \Z_p$
		a section of the canonical projection. If $E$ is a free $\mathcal{O}_X$-module
		of rank $r$ with $\mathscr{D}_{X/k}$-action, and if
		$\overline{E}_1$, $\overline{E}_2$ are free $\tau$-extensions of $E$, via
		$\phi:\overline{E}_1|_X\xrightarrow{\sim} \overline{E}_2|_X$, then there is a
		$\mathscr{D}_{\overline{X}/k}(\log D_X)$-isomorphim $\overline{E}_1\rightarrow
		\overline{E}_2$ extending $\phi$.
	\end{Lemma}
	\begin{proof}
		This argument is an adaption of \cite[Prop 4.7]{Baldassarri}.
		Let $M$ be the free $A$-module corresponding to $\overline{E}_1$ and
		$\overline{E}_2$. Denote by
		\[\nabla_i:\mathscr{D}_{\overline{X}/k}(\log D_X)\rightarrow
	\End_k(M)\] the two $\mathscr{D}_{\overline{X}/k}(\log
	D_X)$-actions on $M$, coming from the actions on $\overline{E}_1, \overline{E}_2$.
	By Proposition \ref{prop:exponentsModZ} we know that $\nabla_1$ and $\nabla_2$ have the same
		exponents $\alpha_1,\ldots,\alpha_r$ along $D_X$.
		Let $s_1,\ldots,
		s_r$ be a basis of $M$ such that
		\[\nabla_1\left(\delta_{x_1}^{(k)}\right)(s_i)\equiv{\alpha_i \choose k}
		s_i + x_1 M,\]
		and let $s'_1,\ldots, s'_r$ be a basis of $M$, such that
		\[\nabla_2\left(\delta_{x_1}^{(k)}\right)(s'_i)\equiv{\alpha_i \choose k}
		s'_i + x_1 M.\]
		We need to check that $\phi(s_i)\in M\subset M\otimes_A A[x_1^{-1}]$ for all $i$.
		Let $k(x_1)$
		be the residue field $A_{(x_1)}/x_1A_{(x_1)}$. Fix $i\in \{1,\ldots, r\}$, and
		write $\phi(s_i)=\sum_{j=1}^r f_{ij}s'_j$ with $f_{ij}\in A[x_1^{-1}]$.
		Assume that there is an integer $\ell_i>0$, such that the maximal pole order of
		the $f_{ij}$ along
		$x_1$ is $\ell_i$. Then $x_1^{\ell_i}\phi(s_i)\in M\subset M\otimes_A
		A[x_1^{-1}]$, and also 
		\[x_1^{\ell_i}\nabla_2\left(\delta_{x_1}^{(k)}\right)(\phi(s_i))\in M\subset M\otimes_A A[x_1^{-1}].\] 
		Tensoring the equality
		\[
		x_1^{\ell_i}\phi\left(\nabla_1\left(\delta_{x_1}^{(k)}\right)(s_i)\right)=x_1^{\ell_i}\nabla_2\left(\delta_{x_1}^{(k)}\right)(\phi(s_i))\]
		with $k(x_1)$, we obtain
		\begin{equation}\label{eqn:uniquenessofextension1}
			x_1^{\ell_i}{\alpha_i \choose k}\sum_{j=1}^mf_{ij}
			s'_j=x_1^{\ell_i}\sum_{j=1}^m
			\nabla_2\left(\delta_{x_1}^{(k)}\right)(f_{ij}s'_j).
		\end{equation}
		In the completion of the discrete valuation ring $A_{(x_1)}$, we can write
		$f_{ij}=\sum_{s=-\ell_i}^{d_{ij}} a_{ijs}x_1^s$, with $x_1$ not dividing in the
		$a_{ijs}$, and $a_{ijs}\neq 0$ for $s=-\ell_i$ and some $j$.  
		But then we can compute
		\[ x_1^{\ell_i}\nabla_2\left( \delta_{x_1}^{(k)}\right)(f_{ij}
		s'_j)=\sum_{s=-\ell_i}^{d_{ij}}a_{ijs}x_1^{s+\ell_i}{\alpha_j+s\choose k}s'_j\]
		Then \eqref{eqn:uniquenessofextension1} gives in $k(x_1)$ the equation
		\begin{align*}
			x_1^{\ell_i+s}a_{ijs}{\alpha_i\choose k}=a_{ijs}x_1^{\ell_i+s}{\alpha_j+s\choose
			k},
		\end{align*}
		both sides of which are $0$, except when $s=-\ell_i$. But this implies
		${\alpha_i\choose k}={\alpha_j-\ell_i\choose k}$ for all $k$, and hence
		$\alpha_i=\alpha_j-\ell_i$ by Lucas' Theorem \ref{lemma:binomials}
		\ref{lemma:lucas}, which is is impossible, as $\alpha_1,\ldots, \alpha_r$
		lie in the image of $\tau$, and thus map injectively to $\Z_p/\Z$. Thus $\ell_i=0$
		and hence $\phi(s_i)\in M$.
	\end{proof}
\end{proof}
\begin{Corollary}\label{cor:essentialimage} The essential image of the fully faithful restriction functor
	\[\Strat(\overline{X})\rightarrow \Strat^{\rs}( (X,\overline{X}))\]
	is the full subcategory of $(X,\overline{X})$-regular singular bundles with
	exponents equal to $0\in \Z_p/\Z$.
\end{Corollary}
\begin{proof}
	By the $\tau$-extension theorem \ref{thm:tauExtension},
	we have to show that a finite rank $\mathcal{O}_{\overline{X}}$-locally free
	$\mathscr{D}_{\overline{X}/k}(\log
	D_X)$-module $\overline{E}$ with exponents $0$ has a canonical
	$\mathscr{D}_{\overline{X}/k}$-action. 
	For this we may assume that $\overline{X}$ is affine with coordinates
	$x_1,\ldots, x_n$ such that $D_X=(x_1)$, and that $\overline{E}$ is free. Then having
	exponents $0$ means that for every $e\in \overline{E}$, and for all $m\geq 1$,
	\[ \delta^{(m)}_{x_1}(e)=0\cdot e + x_1\overline{E}.\]
	Thus we can define $\partial^{(1)}_{x_1}(e):=\frac{\delta^{(1)}_{x_1}(e)}{x_1}$.
	In particular, the $\mathscr{D}_{\overline{X}/k}(\log D_X)$-action defines an honest
	flat connection with $p$-curvature $0$ on $\overline{E}$.
	Then, by Cartier's Theorem (\cite[Thm.~5.1]{Katz/Connections}), if $(-)^{(1)}$ denotes
	Frobenius twist, then 
	$\overline{E}=F_{X/k}^*\overline{E}_1$, where $\overline{E}_1$ is the
	$\mathscr{D}_{\overline{X}^{(1)}/k}(\log D_X^{(1)})$-module obtained as the sheaf of sections $s$
	of $\overline{E}$ such that $\partial^{(1)}_{x_i}(s)=0$ for all $i$.
	Moreover, $\overline{E}_1$ also has exponents $0$, and
	$\delta_{x_1}^{(p)}$ acts as $\delta^{(1)}_{x_1^{(1)}}$ on $\overline{E}_1$. We
	reapply the argument to give meaning to the action of
	$\partial_{x_1^{(1)}}^{(1)}=\partial_{x_1}^{(p)}$ on
	$\overline{E}$. Then we apply Cartier's Theorem again, etc.
\end{proof}
\begin{Remark}
	Corollary \ref{cor:essentialimage} reveals a big difference to the classical
	situation over the complex numbers: If $(X,\overline{X})$ is a good partial
	compactification over $\C$, then  a flat connection on $X$ can be
	$(X,\overline{X})$-regular
	singular with all exponents $0$ (i.e.~with nilpotent residues), but still not
	extend to a flat connection on $\overline{X}$.
\end{Remark}
\section{$(X,\overline{X})$-Regular Singular Stratified Bundles with Finite Monodromy}\label{sec:rsfinite}
We are ready to prove  Main Theorem \ref{thm:mainINTRO} with respect to a fixed good partial
compactification. As before, let $k$ denote an algebraically closed field of
characteristic $p>0$.

\begin{Theorem}[$(X,\overline{X})$-Main Theorem]\label{thm:gpcmain}
	Let $(X,\overline{X})$ be a good partial compactification (Definition
	\ref{defn:gpc}). Then for a stratified bundle $E\in \Strat(X)$, the following
	statements are equivalent:
	\begin{enumerate}
		\item $E$ is $(X,\overline{X})$-regular singular and has finite monodromy.
		\item There exists a finite \'etale covering $f:Y\rightarrow X$, tamely
			ramified with respect to $\overline{X}\setminus X$, such that
			$f^*E\in \Strat(Y)$ is trivial.
	\end{enumerate}
\end{Theorem}

\begin{Remark} Recall that in the situation of Theorem \ref{thm:gpcmain} the morphism $f$ is tamely
	ramified with respect to $\overline{X}\setminus  X$, if the discrete rank $1$
	valuations of $k(X)$ associated with the codimension $1$ points of
	$\overline{X}\setminus X$ are tamely ramified in $k(Y)$.
\end{Remark}

The theorem will be deduced from the following lemma, which is the technical heart of the
proof:

\begin{Lemma}[Main Lemma]\label{lemma:gpcmain} Let $(X,\overline{X})$ be a good partial compactification and
	$f:Y\rightarrow X$ a finite Galois \'etale morphism. Then the stratified
	bundle $f_*\mathcal{O}_Y$ is $(X,\overline{X})$-regular singular, if and only if $f$ is tamely
	ramified with respect to $\overline{X}\setminus X$.
\end{Lemma}

\begin{proof}[Proof of Theorem \ref{thm:gpcmain} (assuming Lemma \ref{lemma:gpcmain})]
	Let $E$ be an $(X,\overline{X})$-regular singular stratified bundle with finite
	monodromy. Let $\omega:\left<E\right>_{\otimes}\rightarrow \Vectf_k$ be a fiber
	functor, and $h_{E,\omega}:X_{E,\omega}=\Spec A_{E,\omega}\rightarrow X$ the Picard-Vessiot torsor
	associated with $E$ and $\omega$ (Definition \ref{defn:pvtorsor}). Then $A_{E,\omega}\in
	\left<E\right>_{\otimes}$ is $(X,\overline{X})$-regular singular, so the Main
	Lemma \ref{lemma:gpcmain} implies that $h_{E,\omega}$ is tamely ramified with
	respect to $\overline{X}\setminus X$. By construction $h_{E,\omega}^{*}E$ is
	trivial.

	Conversely, if $f:Y\rightarrow X$ is finite \'etale and tamely ramified with
	respect to $\overline{X}\setminus X$, then $E\subset
	f_*f^*E=f_*\mathcal{O}_{Y}^{\rank E}$ as stratified bundles, and $f_*\mathcal{O}_Y$ is
	$(X,\overline{X})$-regular singular by the Main Lemma \ref{lemma:gpcmain}.
\end{proof}

Now to the proof of the Main Lemma \ref{lemma:gpcmain}:
\begin{proof}
	Corollary \ref{cor:logpullback} implies that
	$f_*\mathcal{O}_Y$ is $(X,\overline{X})$-regular singular if $f$ is tamely
	ramified with respect to $\overline{X}\setminus X$. Indeed, we may assume that
	$D_X:=\overline{X}\setminus X$ is a smooth divisor and that $\overline{X}=\Spec A$ is
	affine, such that $\overline{X}\setminus X$ is cut out by a regular element $t\in
	A$. Then $Y=\Spec B$ is affine, and we can write $\overline{Y}$ for the
	normalization of $\overline{X}$ in $k(B)$.  After shrinking $\overline{X}$ around the generic point of
	$D_X$ if necessary, and after replacing $\overline{Y}$ by the
	preimage of the smaller $\overline{X}$, $D_Y:=\overline{Y}\setminus Y$ is a strict
	normal crossings divisor, and we get a commutative diagram
        \[\xymatrix{
	\overline{Y}\ar[r]^{\bar{f}}&\overline{X}\\
	Y\ar@{^{(}->}[u]\ar[r]^{f}&X\ar@{^{(}->}[u]}\]
	Then Corollary \ref{cor:logpullback} applies and shows that
	$f_*\mathcal{O}_Y$ is $(X,\overline{X})$-regular singular.

	The converse is more involved:
	Again we may assume without loss of generality that $D_X$ is a
	smooth irreducible divisor with generic point $\eta$, and in the construction we may shrink
	$\overline{X}$ around $\eta$. 
	We proceed in five steps:
	\begin{enumerate}[label=(\alph*), ref=(\alph*)]
		\item Note that the exponents of $f_*\mathcal{O}_Y$ are torsion in
			$\Z_p/\Z$, because by Proposition
			\ref{prop:pullbackmultipliesexponents}, pulling back
			an $\mathscr{D}_{\overline{X}/k}(\log \overline{X}\setminus X)$-extension of
			$f_*\mathcal{O}_Y$ along $\bar{f}$ multiplies the exponents by the
			ramification indices of $\bar{f}$ along $D_X$, and clearly
			$f^*f_*\mathcal{O}_Y=\mathcal{O}_Y^{\deg f}$ is trivial.
        	\item By Theorem \ref{thm:tauExtension} we find an
			$\mathcal{O}_{\overline{X}}$-coherent,
			$\mathcal{O}_{\overline{X}}$-torsion-free extension
			$\overline{E}$ of $f_*\mathcal{O}_Y$ with 
        		$\mathscr{D}_{\overline{X}/k}(\log D_X)$-action and exponents in $\Z\cap \Q$; say
        		$\frac{a_1}{b},\ldots, \frac{a_r}{b}$ with $(b,p)=1$.
        	\item Shrinking $\overline{X}$ around $\eta$, if necessary, we may assume
        		that $\overline{X}=\Spec A$, with local coordinates $x_1,\ldots,
        		x_n$ such that $D_X=(x_1)$. Then define $\overline{Z}_1:=\Spec
        		A[x_1^{1/b}]$, and let $\bar{h}:\overline{Z}_1\rightarrow
        		\overline{X}$ be the associated covering. Let
        		$Z_1:=\bar{h}^{-1}(X)$ and $h=\bar{h}|_{Z_1}$. Then $h$ is
			\'etale,
        		$\bar{h}$ finite and tamely ramified with respect to
        		$\overline{X}\setminus X$, and $h^*f_*\mathcal{O}_Y$ has
        		exponents equal to $0$ in $\Z_p/\Z$, which means that by
        		Corollary \ref{cor:essentialimage} there exists a stratified bundle
        		$\overline{E}_1\in \Strat(\overline{Z}_1)$ extending
        		$h^*f_*\mathcal{O}_Y$.
        	\item Now we claim that there exists a finite \'etale covering
        		$\bar{g}:\overline{Z}\rightarrow\overline{Z}_1$ such that
        		$\bar{g}^*\overline{E}_1$ is trivial. Indeed, this is true for
			$\overline{E}_1|_{Z_1}=h^*f_*\mathcal{O}_Y$ because it is true for $f_*\mathcal{O}_Y$,
        		and by Proposition \ref{prop:restriction} restriction
        		functor $\left<\overline{E}_1\right>_{\otimes}\rightarrow
        		\left<h^*f_*\mathcal{O}_Y\right>_{\otimes}$ is an equivalence.        
		\item We can finish up: Write $Z:=\bar{g}^{-1}(Z_1)$,
        		$g=\bar{g}|_{Z}$ and $h=h_{1}g$. Then we have the following
        		diagram of finite \'etale maps
        		\[\xymatrix{
        		Y\times_X Z\ar[dd]_{h_Y}\ar[r]^--{f_Z}&
        		Z\ar[d]^{g}\ar@/^1cm/[dd]^{h}\\
        		& Z_1\ar[d]^{h_1}\\
        		Y\ar[r]^{f}& X}\]
        		and $h$ is tamely ramified with respect to $\overline{X}\setminus
        		X$ by construction. But also by construction $h^*f_*\mathcal{O}_Y$
        		is trivial, and since
			$h^*f_*\mathcal{O}_Y=f_{Z,*}\mathcal{O}_{Y\times_X Z}$,
			Corollary \ref{cor:trivialDmod} shows that $f_Z$ is the trivial
			covering. But this means that the covering $h:Z\rightarrow X$
			dominates $f:Y\rightarrow X$, so $f$ is tamely ramified with
			respect to $\overline{X}\setminus X$.
	\end{enumerate}
\end{proof}

Now write $D_X:=\overline{X}\setminus X$ and denote by $\pi_1^{D_X}(X,x)$ the profinite group
associated with the Galois category of finite \'etale coverings of $X$ which are
tamely ramified with respect to $D_X$.
\begin{Corollary}
	Let $(X,\overline{X})$ be a good partial compactification, and
	$D_X:=\overline{X}\setminus X$. Let $x\in X(k)$ be a
	rational point. Then the fiber functor $\omega_x:\Strat^{\rs}(
	(X,\overline{X}))\rightarrow
	\Vectf_k$ induces an equivalence of the category of $(X,\overline{X})$-regular
	singular stratified bundles with finite monodromy with the category
	$\Repf_{k}^{\cont}\pi_1^{D_X}(X,x)$. 

	In other words: 	If $\pi_1^{D_X}(X,x)=\varprojlim_i G_i$ with $G_i$ finite, then the maximal pro-\'etale quotient of $\pi_1(\Strat^{\rs}(
	(X,\overline{X})),\omega_x)$ is $\pi_1^{D_X}(X,x)_k:=\varprojlim_i (G_i)_k$, where
	$\omega_x$ is the neutral fiber functor associated with $x$.
\end{Corollary}

\section{Regular Singular Stratified Bundles in General}\label{sec:rsingeneral}
Since resolution of singularities is not available in positive characteristic, we
unfortunately cannot use good compactifications to define regular singularity of
stratified bundles in positive characteristic.  In this section we present a definition
which
works in general, and we generalize the results from the previous sections to this new
notion of regular singularity.

We continue to denote by $k$ an algebraically closed field of characteristic $p>0$, and by
$X$ a smooth, connected, separated $k$-scheme of finite type.

\begin{Definition}
	A stratified bundle $E$ on $X$ is called \emph{regular singular} if it is
	$(X,\overline{X})$-regular singular for all good partial compactifications
	$(X,\overline{X})$ of $X$. The category $\Strat^{\rs}(X)$ is defined to be the
	full subcategory of $\Strat(X)$ with objects the regular singular stratified
	bundles.
\end{Definition}

This is inspired by the following definition:
\begin{Definition}[Kerz-Schmidt, Wiesend, \cite{Kerz/tameness}]A finite \'etale morphism $f:Y\rightarrow X$ 
	is called \emph{tame} if the induced extension $k(X)\hookrightarrow k(Y)$ is
	tamely ramified with respect to every \emph{geometric discrete rank $1$ valuation
	of $K(X)$}. Here a discrete rank $1$ valuation of $K(X)$ is called
	\emph{geometric} if its valuation ring appears as the local ring of a codimension
	$1$ point on some model of the function field $K(X)$ of $X$.
\end{Definition}

\begin{Remark}
	It is not difficult to see that $f:Y\rightarrow X$ is tame if and only if it is
	tamely ramified with respect to $\overline{X}\setminus X$ for all good partial
	compactifications $(X,\overline{X})$ of $X$.
\end{Remark}

Proposition \ref{prop:rs-tannakian} immediately implies:
\begin{Proposition}
	The category $\Strat^{\rs}(X)$ is a sub-Tannakian subcategory of $\Strat(X)$.
\end{Proposition}

Lets see that the above definition of regular singularity agrees with the one from
\cite{Gieseker/FlatBundles} in the presence of a good compactification:
\begin{Proposition}
	Assume that $X$ admits a good compactification $\overline{X}$, i.e.~a good partial
	compactification $(X,\overline{X})$ with $\overline{X}$ \emph{proper}. Then a
	stratified bundle $E$ is regular singular, if and only if it is
	$(X,\overline{X})$-regular singular.
\end{Proposition}
\begin{proof}
	If $E$ is regular singular, then $E$ is $(X,\overline{X})$-regular
	singular by definition, so it remains to prove the converse. Assume that $E$ is
	$(X,\overline{X})$-regular singular, and  let $(X,\overline{X}')$ be any good
	partial compactification. To show that $E$ is $(X,\overline{X}')$-regular singular,
	we may remove a closed subset of codimension $\geq 2$ from
	$\overline{X}'$. Hence, by
	the properness of $\overline{X}$ we
	may assume that there exists a
	morphism $g:\overline{X}'\rightarrow \overline{X}$, such that $g|_X=\id_X$. Then
	$g$ satisfies the assumptions of Proposition \ref{prop:logfunctoriality}, and
	thus if  $\overline{E}$ is an  $\mathcal{O}_{\overline{X}}$-torsion-free,
	$\mathcal{O}_{\overline{X}}$-coherent $\mathscr{D}_{\overline{X}/k}(\log
	\overline{X}\setminus X)$-module $\overline{E}$ extending $E$, then
	$g^*\overline{E}$ is an $\mathcal{O}_{\overline{X}'}$-torsion-free,
	$\mathcal{O}_{\overline{X}'}$-coherent $\mathscr{D}_{\overline{X}'/k}(\log
	\overline{X}'\setminus X)$-module extending $E$, so $E$ is
	$(X,\overline{X}')$-regular singular.
\end{proof}

The proof of the Main Theorem \ref{thm:mainINTRO} now is simple:
\begin{Theorem}\label{thm:main}
	For a stratified bundle $E\in \Strat(X)$, the following statements are equivalent:	
	\begin{enumerate}
		\item $E$ is regular singular and has finite monodromy.
		\item $E$ is trivialized by a finite \'etale tame morphism.
	\end{enumerate}
\end{Theorem}
\begin{proof}
	The stratified bundle $E$ is regular singular with finite monodromy, if and only
	if $E$ is
	$(X,\overline{X})$-regular singular for every good partial compactification, which
	by
	Theorem \ref{thm:gpcmain} is the case if and only if  $f$ is tamely ramified with respect to
	$\overline{X}\setminus X$ for every good partial compactification, i.e.~if and
	only if $f$ is tame.
\end{proof}

Now write $\pi_1^{\tame}(X,x)$ for the profinite group associated with the Galois
category of all finite \'etale tame coverings of $X$, as defined in
\cite{Kerz/tameness}.
\begin{Corollary}
	Let $x\in X(k)$ be a rational point, and $\omega_x:\Strat^{\rs}(X)\rightarrow
	\Vectf_k$ the associated fiber functor. Then $\omega_x$ induces an equivalence of
	the full subcategory of $\Strat^{\rs}(X)$ with objects the regular singular
	stratified bundles with finite monodromy with the category
	$\Repf_{k}^{\cont}\pi_1^{\tame}(X,x)$.

	In other words, if $\pi_1^{\tame}(X,x)=\varprojlim_i G_i$, then the maximal
	pro-\'etale quotient of $\pi_1(\Strat^{\rs}(X),\omega_x)$ is
	$\pi_1^{\tame}(X,x)_k:=\varprojlim_i (G_i)_k$.
\end{Corollary}

\section{Testing for Regular Singularities on Curves}\label{sec:curves}
If $X$ is a smooth complex variety, then one of the basic facts about a flat connection
$(E,\nabla)$ on
$X$ is that $(E,\nabla)$ is regular singular, if and only if for all regular $\C$-curves
$C$ and all $\C$-morphisms $\phi:C\rightarrow X$, the flat connection $\phi^*(E,\nabla)$ on $C$ is
regular singular. We prove an analogue for stratified bundles with finite monodromy with
respect to our notion of regular singularity in
positive characteristic. For information about stratified bundles with arbitrary
monodromy, see Remark \ref{rem:curves}.

Let $k$ denote an algebraically closed field of characteristic $p>0$. Again we first work with respect to a fixed good partial compactification.

We start with a lemma:
\begin{Lemma}\label{lemma:pvoncurves}
	Let $X$ be a smooth, separated, finite type $k$-scheme, $E$ a stratified bundle
	on $X$ with finite monodromy, and $\omega:\left<E\right>_{\otimes}\rightarrow \Vectf_k$ a fiber functor. 
	Let $\phi:C\rightarrow X$ be a nonconstant morphism with $C$ are regular $k$-curve.
	Then the following statements are true:
	\begin{enumerate}
		\item There exists a fiber functor
			$\omega_\phi:\left<E|_C\right>_{\otimes}\rightarrow \Vectf_k$, such
			that the diagram
			\[\xymatrix{
			\left<E\right>_{\otimes}\ar[d]_{\text{\emph{restriction}}}\ar[r]^{\omega}&\Vectf_k\\
			\left<E|_C\right>_{\otimes}\ar[ur]_{\omega_\phi}}
			\]
			commutes. Here $E|_C:=\phi^*E$.
		\item If $h_{E,\omega}:X_{E,\omega}\rightarrow X$ is the Picard-Vessiot
			torsor associated with $E$ and $\omega$, then
			$h_{E,\omega}\times_X \phi:X_{E,\omega}\times_X C\rightarrow C$ is isomorphic to the disjoint union of
			copies of the Picard-Vessiot torsor
			$h_{E|_C,\omega_\phi}:C_{E|_C,\omega_\phi}\rightarrow C$.
	\end{enumerate}
\end{Lemma}
\begin{proof}
	Recall that $G(E,\omega)$ is a finite, constant $k$-group scheme by Theorem
	\ref{thm:dosSantos}; write
	$G$ for the associated finite group.	Then $h_{E,\omega}$ is Galois \'etale with group $G$.
	Consider the fiber product $X_{E,\omega}\times_{X} C\rightarrow C$. This is a
	finite \'etale covering, and it is the disjoint union $\coprod_{j} C_j\rightarrow
	C$, with $C_j$ connected regular curves. Moreover,  all $C_j$ are $C$-isomorphic,
	say to $f:C'\rightarrow C$, as they are permuted by the action of $G$. To
	summarize notation, we have the following commutative diagram
	\begin{equation}
		\begin{split}\label{eq:Cdiag}
		\xymatrix{ C'\ar@{^{(}->}[r]^--i\ar[dr]_{f}&X_{E,\omega}\times_X C=\coprod
		C'\ar[r]^--{\pr}\ar[d]&X_{E,\omega}\ar[d]^{h_{E,\omega}}\\
		&C\ar[r]^{\phi}& X,
		 }
	 \end{split}
	 \end{equation}
	where $i$ is one of the natural inclusions $C'\hookrightarrow\coprod C'$.
	Define $\omega_\phi:\left<E|_C\right>_{\otimes}\rightarrow \Vect_k$ by $F\mapsto
	H^0(\Strat(C'),f^*F)$, see Remark \ref{rem:tannaka}. This is a
	$k$-linear fiber functor since $f^*(E|_C)$ is trivial, and we obtain a
	commutative diagram
	\begin{equation*}
		\xymatrix{ 
		\left<E\right>_{\otimes}\ar[rr]^{|_C} \ar[dr]_{\omega}&&
		\left<E|_C\right>_{\otimes}\ar[dl]^{\omega_\phi}\\
		& \Vect_k
		 }
	 \end{equation*}
	 Indeed, again by Proposition \ref{prop:universaltorsor}, we know that (with notations
	 from \eqref{eq:Cdiag})
	 \begin{align*}\omega(N)&=H^0(\Strat(X_{E,\omega}),\underbrace{h_{E,\omega}^*N}_{\text{trivial}})\\
		 &=H^0(\Strat(C'),i^*\pr^*h_{E,\omega}^*N)\\
		 &=H^0(\Strat(C'),f^*N|_C)\\
		 &=\omega_{\phi}( N|_C)
	 \end{align*}
	 for every object $N\in \left<E\right>_{\otimes}$.
	 By \cite[Prop.
	 2.21b]{DeligneMilne} this implies that $G(\left<E|_C\right>_{\otimes},\omega_{{\phi}})\hookrightarrow
	 G(\left<E\right>,\omega)$ is a closed immersion, and that $f:C'\rightarrow C$ is a
	 $G(\left<E|_C\right>,\omega_{{\phi}})$-torsor; in fact it is the Picard-Vessiot
	 torsor associated with $E|_C$ and $\omega|_C$, according to the following elementary lemma:
	 \begin{Lemma}
		 Let $H\subset G$ be finite groups, and $R\subset G$ be a set of representatives
		 for $G/H$. Then $k[G]=\bigoplus_{r\in R} k[H]$ in the category of $k[H]$-modules.
	 \end{Lemma}
 \end{proof}
From Lemma \ref{lemma:pvoncurves} and the results of \cite{Kerz/tameness} it follows that regular singularity
(at least for stratified bundles with finite monodromy) is a property determined on the
``2-skeleton'' of $X$:
\begin{Theorem}\label{thm:curvesInGeneral}
	Let $X$ be a smooth, finite type $k$-scheme,
	Then a stratified bundle $E$ on
	$X$ with finite monodromy is regular singular, if and only if $E|_C:=\phi^*E$ is
	regular singular for every $k$-morphism $\phi:C\rightarrow X$, with $C$ a regular
	$k$-curve.
\end{Theorem}
\begin{proof}
	Let $\omega$ be a neutral fiber functor for $\left<E\right>_{\otimes}$, and
	let $h_{E,\omega}:X_{E,\omega}\rightarrow X$ be the Picard-Vessiot torsor
	for $E$ and $\omega$. Clearly	$E$ is regular singular if and only if
	$(h_{E,\omega})_*\mathcal{O}_{X_{E,\omega}}$ is regular singular, which by Theorem
	\ref{thm:main} is equivalent to $h_{E,\omega}$ being tame. By
	\cite[Thm. 4.4]{Kerz/tameness}, $h_{E,\omega}$ is tame if and only if
	$h_{E,\omega}\times_X \phi:X_{E,\omega}\times_X C\rightarrow C$ is tame for all
	$\phi:C\rightarrow X$ as in the claim. But by Lemma \ref{lemma:pvoncurves},
	$h_{E,\omega}\times_X \phi$ is isomorphic to (a disjoint union of copies of) the
	Picard-Vessiot torsor associated with $E|_C$ and the fiber functor $\omega_\phi$
	constructed in Lemma \ref{lemma:pvoncurves}.
	This shows that $E$ is regular singular if and only if $h_{E|_C,\omega_\phi}$ is
	tame for  all $\phi:C\rightarrow X$,  if and only if $E|_C$
	is regular singular for  all $\phi:C\rightarrow X$.
\end{proof}

\begin{Remark}\label{rem:curves}
	It is unknown to the author whether Theorem \ref{thm:curvesInGeneral} remains true
	without the finiteness assumption on the monodromy of $E$. There are partial
	results assuming resolution of singularities, see \cite[Sec.~3.4]{Kindler/thesis}.
\end{Remark}
\input{final-submission-kindler-v2.bbl}


\end{document}

%% file: final-submission-kindler-v2.bbl
\providecommand{\bysame}{\leavevmode\hbox to3em{\hrulefill}\thinspace}
\providecommand{\MR}{\relax\ifhmode\unskip\space\fi MR }
\providecommand{\MRhref}[2]{%
  \href{http://www.ams.org/mathscinet-getitem?mr=#1}{#2}
}
\providecommand{\href}[2]{#2}